\documentclass[12pt]{amsart}
\usepackage{amsmath}
\usepackage{amssymb}
\usepackage{verbatim}
\usepackage{float}
\allowdisplaybreaks[4]

\newtheorem{theorem}{Theorem}[section]
\newtheorem{lemma}[theorem]{Lemma}
\newtheorem{proposition}[theorem]{Proposition}

 \theoremstyle{definition}
\newtheorem{definition}[theorem]{Definition}

\theoremstyle{remark}
\newtheorem{remark}[theorem]{Remark}

\numberwithin{equation}{section}

\begin{document}

\title[Neumann problem]
{The Neumann problem for a class of degenerate Hessian quotient type equations }

\author{Jiabao Gong}
\address{Faculty of Mathematics and Statistics, Hubei Key Laboratory of Applied Mathematics, Hubei University,  Wuhan 430062, P.R. China}
\email{202321104011284@stu.hubu.edu.cn}

\author{Qiang Tu$^{\ast}$}
\address{Faculty of Mathematics and Statistics, Hubei Key Laboratory of Applied Mathematics, Hubei University,  Wuhan 430062, P.R. China}
\email{qiangtu@hubu.edu.cn}


\thanks{This research was supported by funds from Natural Science Foundation of Hubei Province, China , No. 2023AFB730 and the National Natural Science Foundation of China No. 12101206.}
\thanks{$\ast$ Corresponding author}

\begin{abstract}
In this paper, we obtain some important inequalities for a class of Hessian quotient type operators $\frac{\sigma_k(\Lambda(D^2u))}{\sigma_l(\Lambda(D^2u))}$, which can be regarded as a generalization of  the classical Hessian quotient operators. As an application, we establish  global a priori estimates  and prove an existence theorem for the Neumann problem of the corresponding degenerate Hessian quotient type equation,
in which the admissible range of $k$  is extended to  $0< k \leq C^\mathbf{p}_n$ with $1 \leq \mathbf{p} \leq n-1$.
\end{abstract}

\maketitle
\section{Introduction}

Let $\Omega\subset\mathbb{R}^n$ be a domain with $C^4$ boundary, $u$ be a $C^2$ function on $\Omega$ and $\lambda=(\lambda_1,\cdots,\lambda_n)$ denote the eigenvalues of the Hessian matrix $D^2u$.
Given an integer $\mathbf{p}$ with $1\leq \mathbf{p}\leq n$, we define the following self-adjoint operator acting on the real vector space $\Lambda^\mathbf{p}\mathbb{R}^n$ as in Harvey-Lawson \cite{HL-2013}:
\begin{equation*}
\begin{aligned}
\mathcal{D}_{u}:&~~~~~~~\quad\Lambda^\mathbf{p}\mathbb{R}^n &&\longrightarrow \Lambda^\mathbf{p}\mathbb{R}^n \\
&v_1\wedge\cdots\wedge v_\mathbf{p} &&\longmapsto \sum_{i=1}^{\mathbf{p}} v_1 \wedge \cdots \wedge v_{i-1}\wedge ((D^2u)v_i)\wedge v_{i+1}  \cdots\wedge v_\mathbf{p},
\end{aligned}
\end{equation*}
where $v_1,\cdots,v_\mathbf{p}\in \mathbb{R}^n$.
For any
$$I=(i_1,\cdots,i_\mathbf{p})\in \mathfrak{J}(\mathbf{p},n):=\{(i_1,\cdots,i_\mathbf{p})|1\leq i_1<\cdots<i_\mathbf{p}\leq n\},$$
we define
$$
\Lambda_I(D^2 u)=\lambda_{i_1}+\cdots+\lambda_{i_\mathbf{p}}.
$$
For convenience, we  fix an order of the elements in $\mathfrak{J}(\mathbf{p},n)$:
\begin{equation*}
I_1,\cdots,I_N,~~\mbox{with}~~N=C^\mathbf{p}_n=\frac{n!}{\mathbf{p}!(n-\mathbf{p})!}.
\end{equation*}
Then the vector of eigenvalues of \(\mathcal D_u\) is given by
$$\Lambda(D^2u)=(\Lambda_{I_1}(D^2u),\cdots,\Lambda_{I_N}(D^2u)).$$

In this paper, we consider the following degenerate Hessian quotient type equation
\begin{equation}\label{(1.1)}
\frac{\sigma_k(\Lambda(D^2 u))}{\sigma_l(\Lambda(D^2 u))}=f(x), \quad \mbox{in}~ \Omega,
\end{equation}
where $0\leq l<k\leq N$, $\sigma_k$ is the $k$-th elementary symmetric function and $f$ is a nonnegative function. Recall that   G\aa rding's cone $\tilde{\Gamma}_{k}\subset \mathbb{R}^N$ is defined by
\begin{equation*}
\tilde{\Gamma}_{k}=\{\lambda \in \mathbb{R}^N: \sigma_{j}(\lambda)>0,~ \forall ~ 1\leq j \leq k\}.
\end{equation*}
If $\Lambda(D^2 u(x)) \in \tilde{\Gamma}_k$ for any $x\in \Omega$, we say $u$ is a $(\Lambda, k)$-admissible function.

 It is worth noting that  equation \eqref{(1.1)} includes several classical elliptic equations as special cases. For example, it reduces to the classical Hessian quotient equation when $\mathbf{p}=1$ and to the Poisson equation when $\mathbf{p}=n$. In particular,  equation \eqref{(1.1)} becomes the following Monge-Amp\`ere equation
 \begin{equation*}\label{(1.4)-1}
\sigma_n(\Delta uI-\nabla^2 u)=f(x),
\end{equation*}
when  $\mathbf{p}=n-1$, $k=n$ and $l=0$. This type of equation arises in the study of several fundamental geometric problems, including the ``form-type" Calabi-Yau equation introduced by Fu-Wang-Wu \cite{FWW10} and the Gauduchon conjecture in complex geometry, which was resolved by Sz$\acute{e}$kelyhidi-Tosatti-Weinkove \cite{STW17} and  Guan-Nie \cite{GN21}.

For the Dirichlet and Neumann problems for non-degenerate elliptic equations in $\mathbb{R}^n$, many results are well established. A priori estimates and existence results for the Dirichlet and Neumann problems with the Laplace equation can be found in \cite{GT}. The Dirichlet problem for the Monge-Amp\`ere equation was solved independently by Caffarelli-Nirenberg-Spruck \cite{CNS84} and Ivochkina \cite{IN}. Lions-Trudinger-Urbas considered the Neumann problem of the Monge-Amp\`ere equation  in their celebrated paper \cite{LION}. For the $k$-Hessian equation,  the Dirichlet problem was resolved by Caffarelli-Nirenberg-Spruck \cite{CNS85}, while the Neumann problem was solved by Trudinger \cite{ST87} in the case of a ball and later extended to strictly convex domains by Ma-Qiu \cite{MQ}.  For the Hessian quotient equation, the Dirichlet and Neumann problems were addressed  by Trudinger \cite{T95} and Chen-Zhang \cite{CZ}, respectively.

Related results for another class of Hessian quotient type equations have also been extensively studied. Chen-Tu-Xiang~\cite{CTX3} investigated the Dirichlet problem for the following Hessian quotient type equation
\begin{equation}\label{(1.4)}
\frac{\sigma_k(\Delta u I - \nabla^2 u)}{\sigma_l(\Delta u I - \nabla^2 u)} = f(x,u,Du),
\end{equation}
on Riemannian manifolds with $0 \le l < k - 1 \le n - 1$. Chen-Tu-Xiang~\cite{CTX2} also established a Pogorelov-type estimate for equation~\eqref{(1.4)}.
 The corresponding Neumann problem was studied  by Dong-Wei \cite{DW23} and Chen-Dong-Han \cite{CDH}.  Dinew \cite{Dinew} studied  the following $\mathbf{p}$-Monge-Amp\`ere type  equation
\begin{equation}\label{1.5}
\Pi_{1\leq i_1<\cdots<i_\mathbf{p}\leq n} \left(\lambda_{i_1}+\cdots+\lambda_{i_\mathbf{p}}\right) =f(x),
\end{equation}
 and obtained  first and second order interior estimates for $\mathbf{p}$-plurisubharmonic solutions.
 Then Chu-Dinew \cite{CD2023} established Liouville theorems and interior estimates for the $\mathbf{p}$-Monge-Amp\`ere type  equation \eqref{1.5}.
 The corresponding curvature equation has been studied  in
 \cite{Chu20, Dong22, Dong24, HMW2011}.

For the degenerate case,   since the right-hand side is nonnegative, it is more difficult to derive a priori estimates, and consequently relatively few results are available.
Guan-Trudinger-Wang~\cite{GTW99} studied the Dirichlet problem for the degenerate Monge-Amp\`ere equation.  They established a priori estimates independent of $\inf_{\overline{\Omega}} f$ under the condition that $f^{\frac{1}{n-1}} \in C^{1,1}(\overline{\Omega})$, and subsequently proved the existence of a unique convex solution $u \in C^{1,1}(\overline{\Omega})$. For the degenerate $k$-Hessian equation, Ivochkina-Trudinger-Wang \cite{ITW19} established a priori estimates and proved an existence theorem under the condition $f^{\frac{1}{k}} \in C^{1,1}(\overline{\Omega})$. They  conjectured that the condition $f^{\frac{1}{k-1}} \in C^{1,1}(\overline{\Omega})$ was a  sufficient condition for obtaining the existence theorem, but now it is still an open problem.
Subsequently, Wang-Xu \cite{27} proved the same existence result under a weaker condition on $f$ than $f^{\frac{1}{k}} \in C^{1,1}(\overline{\Omega})$.
Furthermore, Mei~\cite{Mei21} proved the existence of a  $C^{1,1}$-solution to the classical Neumann problem for the degenerate elliptic Hessian quotient equation under the condition that $f^{\frac{1}{k-l}} \in C^2(\overline{\Omega})$.

Recently, for equation \eqref{(1.1)},
 the authors \cite{GTL} established the global a priori estimates and proved an existence theorem for the Neumann problem  when  $f$ depends on $x, u$ and $Du$. It is worth noting that the condition $0 \leq l < k \leq C_{n-1}^{\mathbf{p}-1}$ is necessary to obtain a priori estimates, as  equation \eqref{(1.1)} is strictly elliptic precisely under this condition.

Naturally, we are interested in the degenerate Neumann problem for equation \eqref{(1.1)} in the more general case where $0 \leq l < k \leq N$ and $ 1 \leq \mathbf{p} \leq n-1$ . In order to obtain the main results, we introduce the definition of $(\Lambda, k)$-admissible weak solutions.
\begin{definition}\label{def-2}
 A function $u$ $\in$ $C^{0}( \overline {\Omega })$ is called a $(\Lambda, k)$-admissible weak solution of equation \eqref{(1.1)} if
 there exists a sequence of $(\Lambda, k)$-admissible functions $\{u_{m}\}\subset C^{2}(\overline{\Omega})$ such that
\begin{equation*}
u_m\to u\quad\mathrm{in}\:C^0(\overline{\Omega})
\end{equation*}
and
\begin{equation*}
\frac{\sigma_k(\Lambda(D^2 u_m))}{\sigma_l(\Lambda(D^2 u_m))}\to f\quad\mathrm{in}\:L_{\mathrm{loc}}^1(\Omega).
\end{equation*}
\end{definition}

 Based on the above notion, we establish the following existence result for the Neumann problem of the degenerate Hessian quotient type equation \eqref{(1.1)}.

\begin{theorem}\label{T1}
Let $1\leq \mathbf{p}<n$, $0 \leq l < k \leq N$, $\Omega\subset \mathbb{R}^n$ be a $C^4$ strictly convex domain and $\nu$ be the outer unit normal vector of $\partial\Omega$. Suppose that $f(x)$ is a nonnegative function  such that $f^{\frac1{k-l}}\in C^{2}(\overline{\Omega})$ and $\varphi(x)\in C^{3}(\partial\Omega)$. Then there exists a unique constant $c$ such that the Neumann problem
\begin{equation}\label{(1.6)}
\left\{\begin{matrix}
\frac{\sigma_k(\Lambda(D^2 u))}{\sigma_l(\Lambda(D^2 u))}=f(x)&in\:\Omega,\\\\u_\nu=c+\varphi(x)&on\:\partial\Omega,
\end{matrix}\right.
\end{equation}
has a  $(\Lambda, k)$-admissible weak solution $u\in C^{1,1}(\overline{\Omega})$ , which is unique up to a constant.
\end{theorem}

\begin{remark}
For the Neumann problem of Hessian quotient type equation \eqref{(1.6)}, it is readily observed that the equation is invariant under the addition of an arbitrary constant to the solution.  Consequently, a uniform a priori bound on the solutions cannot be established, which precludes the direct application of the method of continuity to prove existence.  To overcome this, we adopt the approximation scheme introduced by Lions-Trudinger-Urbas \cite{LION} (see also Qiu-Xia \cite{QX}), considering the solution $ u_\varepsilon $ of the perturbed equation
\begin{equation}\label{(1.7)}
\left\{\begin{matrix}
\frac{\sigma_k(\Lambda(D^2 u))}{\sigma_l(\Lambda(D^2 u))}=f(x)&in\:\Omega,\\\\u_\nu=-\varepsilon u+\varphi(x)&on\:\partial\Omega,
\end{matrix}\right.
\end{equation}
for each sufficiently small $\varepsilon>0$. We aim to establish uniform a priori estimates for $u_{\varepsilon}$ that are independent of $\varepsilon$ and $\inf_{\overline{\Omega}} f$. With such estimates in hand, the  solution to \eqref{(1.6)} can be obtained by passing to the limit as $\varepsilon\rightarrow0$, combined with a standard perturbation argument.
\end{remark}

Based on the a priori estimates for \eqref{(1.7)}, we can also establish an existence result for the Hessian quotient type equation with Neumann boundary condition in the non-degenerate case.
\begin{theorem}\label{T2}
Let $1\leq \mathbf{p}<n$, $0 \leq l < k \leq N$, $\Omega\subset \mathbb{R}^n$ be a $C^4$ strictly convex domain and $\nu$ be the outer unit normal vector of $\partial\Omega$. Suppose that $f(x)$ is a positive function,  $f^{\frac1{k-l}}\in C^{2}(\overline{\Omega})$ and $\varphi(x)\in C^{3}(\partial\Omega)$. Then there exists a unique $(\Lambda, k)$-admissible  solution $u\in C^{3,\alpha}(\overline{\Omega})$ to the Neumann problem for the Hessian quotient type equation
\begin{equation}\label{(1.8)}
\left\{\begin{matrix}
\frac{\sigma_k(\Lambda(D^2 u))}{\sigma_l(\Lambda(D^2 u))}=f(x)&in\:\Omega,\\\\u_\nu=-u+\varphi(x)&on\:\partial\Omega.
\end{matrix}\right.
\end{equation}
\end{theorem}

\begin{remark}
Compared with the result in~\cite{GTL}, our result is established under a stronger assumption on the domain, but it applies to a broader range of $k$. At present, the strict convexity assumption on the domain appears to be indispensable.
\end{remark}

When $0\leq l<k\leq C_{n-1}^{\mathbf{p}-1}$, the operator $\frac{\sigma_k(\Lambda(\lambda))}{\sigma_l(\Lambda(\lambda))}$ turns out to satisfy the following strong property
$$\frac{\partial \left[\frac{\sigma_k(\Lambda(\lambda))}{\sigma_l(\Lambda(\lambda))}\right]}{\partial \lambda_i} \geq C \sum_i\frac{\partial \left[\frac{\sigma_k(\Lambda(\lambda))}{\sigma_l(\Lambda(\lambda))}\right]}{\partial \lambda_i}$$
for $\Lambda(\lambda) \in \tilde{\Gamma}_k$ (see Lemma 2.4 in \cite{ZJ}). This property enables us to establish a priori estimates  and the existence theorem with very few assumptions  for the Neumann problem of  equation \eqref{(1.1)} in \cite{GTL}. For $0\leq l <k\leq N$, whether the same property continues to hold remains unknown.
Our main contribution in this paper is to show that the similar properties hold true for $\frac{\sigma_k(\Lambda(\lambda))}{\sigma_l(\Lambda(\lambda))}$ under some assumptions (see Lemma \ref{f11}, Lemma \ref{l2}), which is crucial for establishing a priori estimates  for the Neumann problem  \eqref{(1.6)} and \eqref{(1.7)}.

The rest of this paper is organized as follows.
Section 2 presents some properties of $\Lambda_I(D^2 u)$ and establishes two important Lemmas. Section 3 is devoted to the $C^0$ estimates.
Global gradient estimates are derived in Section 4, while the second-order estimates are obtained in Section 5. Finally, in Section 6, we prove the existence theorem.

\section{Preliminaries}
In this section, we recall the definition and some basic properties of elementary symmetric functions, which could be found in~\cite{CCQ, GN}, and establish some key lemmas.

\subsection{Basic properties of elementary symmetric functions}\

Let $\lambda=(\lambda_1,\dots,\lambda_n)\in\mathbb{R}^n$, we recall
the definition of elementary symmetric functions for $1\leq k\leq n$,
\begin{equation*}
\sigma_k(\lambda)= \sum _{1 \le i_1 < i_2 <\cdots<i_k\leq
n}\lambda_{i_1}\lambda_{i_2}\cdots\lambda_{i_k}.
\end{equation*}
We also set $\sigma_0=1$ and $\sigma_k=0$ for $k>n$ or $k<0$. The  G\aa rding's cone is defined by
\begin{equation*}
\Gamma_k  = \{ \lambda  \in \mathbb{R}^n :\sigma _i (\lambda ) >
0,~\forall~ 1 \le i \le k\}.
\end{equation*}
We denote $\sigma_{k-1}(\lambda|i)=\frac{\partial
\sigma_k}{\partial \lambda_i}$ and
$\sigma_{k-2}(\lambda|ij)=\frac{\partial^2 \sigma_k}{\partial
\lambda_i\partial \lambda_j}$. Next, we list some properties of
$\sigma_k$ which will be used later.

\begin{proposition}\label{P1}
Let $\lambda=(\lambda_1,\dots,\lambda_n)\in\mathbb{R}^n$ and $1\leq k\leq n$, then we have
\begin{enumerate}
\item[(\romannumeral1)]  $\Gamma_k $ are convex cones and $\Gamma_1\supset\Gamma_2\supset\cdots \supset\Gamma_n$.
\item [(\romannumeral2)]   $\sigma_{k-1}(\lambda|i)>0$ for $\lambda\in\Gamma_k,$ and $1\leq i\leq n$.
\item [(\romannumeral3)]  $\sigma_k(\lambda)=\sigma_k(\lambda|i)
+\lambda_i\sigma_{k-1}(\lambda|i)$ for $1\leq i\leq n$.
\item [(\romannumeral4)] If $\lambda\in\Gamma_k$  with $\lambda_1\geq \cdots\geq \lambda_k\geq \cdots \geq \lambda_n$, then
$$\lambda_1 \sigma_{k-1}(\lambda |1) \geq \frac{k}{n}\sigma_k(\lambda), \quad \sigma_{k-1}(\lambda|k)\geq C(n,k)\sigma_{k-1}(\lambda),$$
and
$$\sigma_{k-1}(\lambda|1)\leq \sigma_{k-1}(\lambda|2)\leq\cdot\cdot\cdot\leq \sigma_{k-1}(\lambda|n).$$
\item [(\romannumeral5)] Newton-MacLaurin inequality: If $\lambda\in\Gamma_k,~n\geq k> l \geq 0$, $ n\geq r > s \geq 0$, $k\geq r$, $l\geq s$, then
\begin{equation}\label{lem21}
\left[\frac{{\sigma _k (\lambda )}/{C_n^k }}{{\sigma _l (\lambda)}/{C_n^l }}\right]^{\frac{1}{k-l}}
\leq \left[\frac{{\sigma _r(\lambda )}/{C_n^r }}{{\sigma _s (\lambda )}/{C_n^s}}\right]^{\frac{1}{r-s}}. \notag
\end{equation}
\item [(\romannumeral6)]  If $\lambda\in\Gamma_k,~0\leq l<k\leq n,$~then
$$\sum\limits_{i=1}^n \frac{\partial \left[\frac{\sigma_k(\lambda)}{\sigma_l(\lambda)}\right]^{\frac{1}{k-l}}}{\partial \lambda_i}\geq
\left[\frac{C^k_n}{C^l_n}\right]^{\frac{1}{k-l}}.$$

\item [(\romannumeral7)]  If $\lambda\in\Gamma_k,~0\leq l<k\leq n$,~then $\left[\frac{\sigma_k(\lambda)}{\sigma_l(\lambda)}\right]^{\frac{1}{k-l}}$ are concave functions in $\Gamma_k$.
\end{enumerate}
\end{proposition}

 For $\lambda=(\lambda_1,\cdots,\lambda_n)\in \mathbb{R}^n$ and $1\leq \mathbf{p}\leq n$, recall the notation $\Lambda(\lambda)=(\Lambda_{I_1},\cdots,\Lambda_{I_N})$,
$$\Lambda_{I_s}=\sum_{i_j\in I_s} \lambda_i=\lambda_{i_1}+\cdots+\lambda_{i_\mathbf{p}}, \quad \forall~~ I_s=(i_1, \cdots, i_\mathbf{p})\in \mathfrak{J}.$$
Then $\Lambda(\lambda)$ satisfies the following basic properties; see Lemma 2.3--2.5 of \cite{ZJ}.

\begin{proposition}\label{P4}
Let $\lambda=(\lambda_1,\lambda_2,\cdots,\lambda_n)\in \mathbb{R}^n$ with $\lambda_1\geq \lambda_2\geq\cdots\geq \lambda_n$ and $\Lambda(\lambda)\in \tilde{\Gamma}_k$. If
$\Lambda(\lambda)=(\Lambda_{I_1},\cdots,\Lambda_{I_N})$ with $\Lambda_{I_1}\geq \cdots\geq \Lambda_{I_N}$, then
\begin{enumerate}
\item[(\romannumeral1)]  $\Lambda_{I_1}=\lambda_1+\cdots+\lambda_\mathbf{p}$ and~~~$\Lambda_{I_N}=\lambda_{n-\mathbf{p}+1}+\cdots+\lambda_n$.
\item [(\romannumeral2)]   We have
\begin{align*}
&\frac{\partial \bigg[ \frac{\sigma_k(\Lambda(\lambda))}{\sigma_l(\Lambda(\lambda))}\bigg]}{\partial \Lambda_{I_1}}\leq \cdots\leq \frac{\partial \bigg[ \frac{\sigma_k(\Lambda(\lambda))}{\sigma_l(\Lambda(\lambda))}\bigg]}{\partial \Lambda_{I_N}},\quad 0\leq l<k\leq N,
\end{align*}
and
\begin{align*}
&\frac{\partial \bigg[ \frac{\sigma_k(\Lambda(\lambda))}{\sigma_l(\Lambda(\lambda))}\bigg]}{\partial \lambda_{1}}\leq \cdots\leq \frac{\partial \bigg[ \frac{\sigma_k(\Lambda(\lambda))}{\sigma_l(\Lambda(\lambda))}\bigg]}{\partial \lambda_{n}},\quad 0\leq l<k\leq N.
\end{align*}
\item [(\romannumeral3)] If~~$0\leq l<k\leq N$, then
$\bigg[\frac{\sigma_k(\Lambda(\lambda))}{\sigma_l(\Lambda(\lambda))}\bigg]^{\frac{1}{k-l}}$ are concave with respect to $\lambda$ and
\begin{equation*}
\sum\limits_{i=1}^n\frac{\partial\left[\frac{\sigma_k(\Lambda(\lambda))}{\sigma_l(\Lambda(\lambda))}\right]^{\frac{1}{k-l}}}{\partial \lambda_i}\geq {C_\mathbf{p}}:=\mathbf{p}\left(\frac{C_N^k}{C_N^l}\right)^{\frac{1}{k-l}}>0.
\end{equation*}
\end{enumerate}
\end{proposition}

\subsection{Definition through derivations on the exterior algebra}
Let $\mathbf{p} \in \{1, \cdots, n\}$ and $n\geqslant 2$. We use the standard notation for ordered multi-indices
$$ \mathfrak{J}(\mathbf{p},n):=\{I=(i_{1},\cdots,i_{\mathbf{p}})\mid i_{s}~integers~and~1\leqslant i_{1}<\cdots<i_{\mathbf{p}}\leqslant n\}. $$
For convenience, fix an order for the elements in $\mathfrak{J}(\mathbf{p},n)$:
$$I_1, I_2, \cdots, I_N, \quad \mbox{with}~N:=C_n^\mathbf{p}= \frac{n!}{\mathbf{p}!(n-\mathbf{p})!}.$$
Set $\mathfrak{J}(0,n)=\{0\}$ and $|I|=\mathbf{p}$ if $I\in \mathfrak{J}(\mathbf{p},n)$. For $I\in \mathfrak{J}(\mathbf{p},n)$,
\begin{enumerate}
\item[(\romannumeral1)]  $\overline{I}$ is the element in $\mathfrak{J}(n - \mathbf{p},n)$ which complements $I$ in $\{1,2,\cdots,n\}$ in the natural increasing order.
\item [(\romannumeral2)] $I-i$ means the multi-index of length $\mathbf{p}-1$ obtained by removing $i$ from $I$ for any $i \in I$.
\item [(\romannumeral3)] $I+j$ means the multi-index of length $\mathbf{p}+1$ obtained by adding $j$ to $I$  for any $j \notin I$.
\item [(\romannumeral4)] $\sigma(I,J)$ is the sign of the permutation which reorders $(I,J)$ in the natural increasing order for any
multi-index $J$ with $I\cap J=\emptyset$. In particular set $\sigma(\overline{0},0):=1$.
\end{enumerate}

\begin{definition}
For $\mathbf{p} \in \{1, \cdots, n\}$, the $(\mathbf{p}, k)$-cone is defined by
$$\mathcal{P}_{\mathbf{p}, k}=\{\lambda=(\lambda_1, \lambda_2, \cdots, \lambda_n)\in \mathbb{R}^n\mid \sigma_j(\Lambda) >0, \forall 1\leq j\leq k\},$$
where $\Lambda=(\Lambda_{I_1}, \Lambda_{I_2}, \cdots, \Lambda_{I_N}) \in \mathbb{R}^N$,  and
$$\Lambda_I=\lambda_{i_1}+\lambda_{i_2}+\cdots+\lambda_{i_\mathbf{p}}$$
for any $I=(i_1, i_2, \cdots, i_p)\in \mathfrak{J}(\mathbf{p},n)$.
\end{definition}
\begin{remark}
It is worth emphasizing that $\lambda \in \mathcal{P}_{\mathbf{p}, k}$ if and only if $\Lambda \in \tilde{\Gamma}_k$ in $\mathbb{R}^N$.
\end{remark}

Based on the above cone condition, we introduce the following  proposition.
\begin{proposition}
 A function $u$  is called a $(\Lambda, k)$-admissible  solution  if  and only if
 $\lambda(D^2u)(x)\in \mathcal{P}_{\mathbf{p}, k}$ for any $x \in \Omega$.
\end{proposition}

\begin{definition}
Let $\mathbf{p} \in \{1, \cdots, n\}$ and  $A=\{a_{ij}\}$ be a symmetric matrix.
The derivation induced by $A$ on $\Lambda^\mathbf{p}\mathbb{R}^n$ is defined as the linear map
\begin{equation*}
\begin{aligned}
\mathcal{D}_A:&~~~~~~~\Lambda^\mathbf{p}\mathbb{R}^n &&\longrightarrow \Lambda^\mathbf{p}\mathbb{R}^n \\
&v_1\wedge\cdots\wedge v_\mathbf{p} &&\longmapsto \sum_{i=1}^{\mathbf{p}} v_1 \wedge \cdots \wedge v_{i-1}\wedge (Av_i)\wedge v_{i+1}  \cdots\wedge v_\mathbf{p}.
\end{aligned}
\end{equation*}
\end{definition}

 Fix an orthonormal basis $(e_{1},\cdot\cdot\cdot,e_{n})$ of $\mathbb{R}^n$ and the corresponding basis $\{e_{I}\}_{I\in\mathfrak{J}(\mathbf{P},n)}$ of $\Lambda^\mathbf{P}\mathbb{R}^n$, where $e_I:=e_{i_1}\wedge\cdots\wedge e_{i_\mathbf{P}}$ for any $I=(i_1,\cdots i_\mathbf{P}) \in \mathfrak{J}(\mathbf{P},n)$. Obviously, the eigenvalues of    $\mathcal{D}_A$ can be written as
 $\Lambda_{I}=(\Lambda_{I_1}, \Lambda_{I_2}, \cdots, \Lambda_{I_N})$ if  $\lambda=(\lambda_1,\lambda_2,\cdots,\lambda_n)$ are eigenvalues of symmetric matrix $A$.

It is worth emphasizing that given any orthonormal basis of $\mathbb{R}^n$, $\mathcal{D}_A$ has a matrix representation with respect to the induced basis which has components being linear combinations of the entries of $A$.

\begin{proposition}\label{wij}
Let $A=\{a_{ij}\}$ be a symmetric matrix, the corresponding matrix $W:=\{ W_{IJ}\}_{I,J\in \mathfrak{J}(\mathbf{p},n)}$ of linear derivation $\mathcal{D}_A$
in the canonical basis $\{ e_{I_1}, e_{I_2}, \cdots, e_{I_N}\}$ reads
\begin{equation*}\label{propw}
W_{IJ}=
\begin{cases}
\sum_{i\in I}a_{ii}, ~~~&I=J,\\
\sigma(i,I-i)\sigma(j,J-j)a_{ij}, ~~~&I=i+K,J=j+K,|K|=\mathbf{p}-1,i\neq j,\\
0,~~~&\mbox{otherwise}.
\end{cases}
\end{equation*}
and hence we have
\begin{equation*}
\frac{\partial W_{IJ}}{\partial a_{ij}}=
\begin{cases}
1, ~~~&i=j,I=J,i\in I,\\
\sigma(i,I-i)\sigma(j,J-j), ~~~&I=i+K,J=j+K,|K|=\mathbf{p}-1,i\neq j,\\
0,~~~&\mbox{otherwise}.
\end{cases}
\end{equation*}
\end{proposition}

\begin{proof}
For any $ I\in \mathfrak{J}$ with $I=(i_1,i_2,\cdots,i_{\mathbf{p}})$, we have
\begin{align*}
\mathcal{D}_A(e_I)=&\sum _{i_s\in I}e_{i_1}\wedge e_{i_2}\wedge\cdots\wedge e_{i_{s-1}}\wedge Ae_{i_s}\wedge e_{i_{s+1}}\wedge\cdots\wedge e_{i_\mathbf{p}}\\
=&\sum _{i_s\in I}e_{i_1}\wedge e_{i_2}\wedge\cdots\wedge e_{i_{s-1}}\wedge \sum_{m=1}^na_{i_sm}e_m \wedge e_{i_{s+1}}\wedge\cdots\wedge e_{i_\mathbf{p}}\\
=&\sum _{i_s\in I}e_{i_1}\wedge e_{i_2}\wedge\cdots\wedge e_{i_{s-1}}\wedge \Bigg(\sum_{m\notin\{i_1,i_2,\cdots,i_{s-1},i_{s+1},\cdots,i_\mathbf{p}\}}a_{i_sm}e_m \Bigg)\wedge e_{i_{s+1}}\wedge\cdots\wedge e_{i_\mathbf{p}}\\
=&\sum _{i_s\in I}e_{i_1}\wedge e_{i_2}\wedge\cdots\wedge e_{i_{s-1}}\wedge (a_{i_si_s}e_{i_s}+\sum_{m\notin I}a_{i_sm}e_m) \wedge e_{i_{s+1}}\wedge\cdots\wedge e_{i_\mathbf{p}}\\
=&\sum _{i_s\in I}a_{i_si_s}e_{I}+\sum _{i_s\in I}\sum _{m\notin I}a_{i_sm}e_{i_1}\wedge e_{i_2}\wedge\cdots\wedge e_{i_{s-1}}\wedge e_{m}\wedge e_{i_{s+1}}\wedge\cdots\wedge e_{i_\mathbf{p}}\\
=&\sum _{i\in I}a_{ii}e_{I}+\sum _{i\in I}\sum _{j\notin I}a_{ij}\sigma (i,I-i)\sigma(j,I-i)e_{I-i+j}.
\end{align*}
Then the formulas follow from direct calculations.
\end{proof}

\begin{proposition}
Suppose $A=\{a_{ij}\}$ is diagonal,  $W:=\{ W_{IJ}\}_{I,J\in \mathfrak{J}(\mathbf{p},n)}$ is the matrix of the linear derivation $\mathcal{D}_A$ in the canonical basis $\{ e_{I_1}, e_{I_2}, \cdots, e_{I_N}\}$ and $m$ is a positive integer, then
\begin{equation*}
\frac{\partial \sigma_m(W) }{\partial a_{ij}}=
\begin{cases}
\sum_{\{I|i\in I\}}\sigma_{m-1}(W | I),    \quad  & i=j,\\
0, \quad & \mbox{otherwise}.
\end{cases}
\end{equation*}
\end{proposition}

\subsection{Key lemmas}

For convenience, we introduce the following notations:
 $$A = \{ I \in \mathfrak{J}(\mathbf{p},n) | 1 \in I \},\quad B = \{ I \in \mathfrak{J}(\mathbf{p},n) | 1 \notin I \}.$$
Motivated by the work of Chen-Zhang~\cite{CZ}, we establish the following fundamental inequalities and properties for the operator $\frac{\sigma_k(\Lambda(\lambda))}{\sigma_l(\Lambda(\lambda))}$, which are essential to the subsequent derivation of the a priori estimates.

\begin{lemma}\label{f11}
Let $1\leq \mathbf{p} < n$, $2\leq k\leq N$, $\lambda=(\lambda_1,\lambda_2,\cdots,\lambda_n)\in \mathcal{P}_{\mathbf{p}, k}$, and $\lambda_1<0$. Then
\begin{eqnarray}\label{xj-1}
\frac{\partial \left[\frac{\sigma_k(\Lambda(\lambda))}{\sigma_l(\Lambda(\lambda))}\right]}{\partial \lambda_1} \geq C_1 \sum_i\frac{\partial \left[\frac{\sigma_k(\Lambda(\lambda))}{\sigma_l(\Lambda(\lambda))}\right]}{\partial \lambda_i}, \quad \forall~~~0\leq l <k \leq N,
\end{eqnarray}
where $C_1$ is a constant depending only on $n, k, l, \mathbf{p}$.
\begin{proof}
For $\mathbf{p}=1$, \eqref{xj-1} holds by Lemma 2.5 in \cite{CZ}.
In the following, we assume $\mathbf{p}>1$ and divide into two cases to prove \eqref{xj-1}.

\textbf{Case 1: } There exists $ I' \in A$  such that $\Lambda_{I'} < 0$.

In this case, we can get directly from Lemma 2.5 in \cite{CZ},
\begin{align*}
\frac{\partial \left[\frac{\sigma_k(\Lambda(\lambda))}{\sigma_l(\Lambda(\lambda))}\right]}{\partial \lambda_1} =\sum_{I\in A} \frac{\partial \left[\frac{\sigma_k(\Lambda(\lambda))}{\sigma_l(\Lambda(\lambda))}\right]}{\partial \Lambda_I}
\geq& \frac{\partial \left[\frac{\sigma_k(\Lambda(\lambda))}{\sigma_l(\Lambda(\lambda))}\right]}{\partial \Lambda_{I'}}\\
\geq& \frac{N(k-l)}{k(N-l)(N-k+1)}  \sum_{I}\frac{\partial \left[\frac{\sigma_k(\Lambda(\lambda))}{\sigma_l(\Lambda(\lambda))}\right]}{\partial \Lambda_I} \\
=& \frac{N(k-l)}{k(N-l)(N-k+1)} \frac{1}{\mathbf{p}} \sum_{i}\frac{\partial \left[\frac{\sigma_k(\Lambda(\lambda))}{\sigma_l(\Lambda(\lambda))}\right]}{\partial \lambda_i},
\end{align*}
where $N=C_n^\mathbf{p}$.

\textbf{Case 2: } $\Lambda_{I} \geq 0$ for all $I \in A$.

In this case, it suffices to prove that
$$\sum_{I \in A} \frac{\partial \left[\frac{\sigma_k(\Lambda(\lambda))}{\sigma_l(\Lambda(\lambda))}\right]}{\partial \Lambda_I}  \geq \frac{1}{C_{n-1}^{\mathbf{p}}+1} \sum_{I}\frac{\partial \left[\frac{\sigma_k(\Lambda(\lambda))}{\sigma_l(\Lambda(\lambda))}\right]}{\partial \Lambda_I},$$
the above inequality is equivalent to
\begin{equation}\label{xj-2}
\sum_{I \in A}\frac{\partial \left[\frac{\sigma_k(\Lambda(\lambda))}{\sigma_l(\Lambda(\lambda))}\right]}{\partial \Lambda_I}  \geq \frac{1}{C_{n-1}^{\mathbf{p}}} \sum_{I \in B} \frac{\partial \left[\frac{\sigma_k(\Lambda(\lambda))}{\sigma_l(\Lambda(\lambda))}\right]}{\partial \Lambda_I} .
\end{equation}
For any $\tilde{I}=(\tilde{i}_2, \tilde{i}_3, \cdots, \tilde{i}_\mathbf{p}) \in C := \{ \tilde{I} \in \mathfrak{J}(\mathbf{p}-1,n) | 1 \notin \tilde{I} \}$, we observe that
$$\Lambda_{1 + \tilde{I}} = \lambda_1 + \lambda_{i_2} + \cdots + \lambda_{i_\mathbf{p}} \geq 0,$$
since $1 + \tilde{I} \in A$. Then
$$ \lambda_{\tilde{i}_2} + \cdots + \lambda_{\tilde{i}_\mathbf{p}} \geq -\lambda_1 > 0.$$
It implies that there exists $\tilde{i}_s \in \tilde{I}$ such that $\lambda_{\tilde{i}_s} > 0$. Hence, for any $I \in B$, there exists $ i_s \in I$ with $\lambda_{i_s} > 0$.
Fix $ I_B \in B$, without loss of generality, assume $\lambda_{i_1} > 0$, then we obtain
$$\Lambda_{I_B} = \lambda_{i_1} + \lambda_{i_2} + \cdots + \lambda_{i_\mathbf{p}}>\Lambda_{\bar{I}}:=\lambda_{1} + \lambda_{i_2} + \cdots + \lambda_{i_\mathbf{p}},$$
where $\bar{I} = (1, i_2, ..., i_\mathbf{p}) \in A$.
In view of Proposition \ref{P4} (\romannumeral1),
$$\sum_{I \in A} \frac{\partial \left[\frac{\sigma_k(\Lambda(\lambda))}{\sigma_l(\Lambda(\lambda))}\right]}{\partial \Lambda_I}  \geq\frac{\partial \left[\frac{\sigma_k(\Lambda(\lambda))}{\sigma_l(\Lambda(\lambda))}\right]}{\partial \Lambda_{\bar{I}}}  \geq \frac{\partial \left[\frac{\sigma_k(\Lambda(\lambda))}{\sigma_l(\Lambda(\lambda))}\right]}{\partial \Lambda_{I_B}} .$$
Summing over all $I_B\in B$, and using the fact that $|B|=C_{n-1}^\mathbf{p}$, we obtain
$$ C_{n-1}^\mathbf{p}\sum_{I \in A} \frac{\partial \left[\frac{\sigma_k(\Lambda(\lambda))}{\sigma_l(\Lambda(\lambda))}\right]}{\partial \Lambda_I}  \geq \sum_{I \in B} \frac{\partial \left[\frac{\sigma_k(\Lambda(\lambda))}{\sigma_l(\Lambda(\lambda))}\right]}{\partial \Lambda_I} .$$
So \eqref{xj-2} holds and the proof is complete.
\end{proof}
\end{lemma}


\begin{lemma}\label{l2}
Let $1\leq \mathbf{p} <n$, $2\leq k\leq N$, $\lambda=(\lambda_1,\lambda_2,\cdots,\lambda_n)\in \mathcal{P}_{\mathbf{p}, k}$. Suppose that $\lambda_2 \geq \lambda_3\geq \cdots \geq \lambda_n$, $\lambda_1>0$, $\lambda_n<0$, $\lambda_1 \geq \delta \lambda_2$ and $-\lambda_n\geq \epsilon \lambda_1$ for small positive constants $\delta$ and $\epsilon$. Then
\begin{equation}\label{xj-11}
\frac{\partial \left[\frac{\sigma_k(\Lambda(\lambda))}{\sigma_l(\Lambda(\lambda))}\right]}{\partial \lambda_1} \geq C_2 \sum_i\frac{\partial \left[\frac{\sigma_k(\Lambda(\lambda))}{\sigma_l(\Lambda(\lambda))}\right]}{\partial \lambda_i}, \quad \forall~~ 0\leq l <k \leq N,
\end{equation}
where $C_2$ is  a constant depending only on $n, k, l, \mathbf{p}, \delta$ and $\epsilon$.
\end{lemma}
\begin{proof}
For $\mathbf{p}=1$, \eqref{xj-11} holds by Lemma 2.7 in \cite{CZ}.
In the following, we assume $\mathbf{p}>1$.  For convenience, fix an order  $I_1, I_2, \cdots, I_N$  for the elements in $\mathfrak{J}(\mathbf{p}, n)$ with
$$I_1=(1,2,\cdots,\mathbf{p}),~I_2=(2,\cdots,\mathbf{p},\mathbf{p}+1),~\cdots,~ I_N=(n-\mathbf{p}+1,\cdots,n).$$
The proof splits into two cases.

\textbf{Case 1: } $\Lambda_{I_N} \geq  -\frac{\epsilon}{2} \lambda_1$. Set $I'=(1,n-\mathbf{p}+2,\cdots,n)\in A$.

\textbf{Subcase 1.1: } $\lambda_1<\lambda_{n-\mathbf{p}+1}$.

In this case,  $\Lambda_{I'}\leq \Lambda_I$ for any $I \in \mathfrak{J}(\mathbf{p},n)$. By Proposition \ref{P4} (\romannumeral1),
$$\frac{\partial \left[\frac{\sigma_k(\Lambda(\lambda))}{\sigma_l(\Lambda(\lambda))}\right]}{\partial \lambda_1}=\sum_{I\in A}\frac{\partial \left[\frac{\sigma_k(\Lambda(\lambda))}{\sigma_l(\Lambda(\lambda))}\right]}{\partial \Lambda_I} \geq \frac{\partial \left[\frac{\sigma_k(\Lambda(\lambda))}{\sigma_l(\Lambda(\lambda))}\right]}{\partial \Lambda_{I'}} \geq\frac{1}{C_n^\mathbf{p}}\sum_I\frac{\partial \left[\frac{\sigma_k(\Lambda(\lambda))}{\sigma_l(\Lambda(\lambda))}\right]}{\partial \Lambda_I} =\frac{1}{\mathbf{p}C_n^\mathbf{p}}\sum_i\frac{\partial \left[\frac{\sigma_k(\Lambda(\lambda))}{\sigma_l(\Lambda(\lambda))}\right]}{\partial \lambda_i}.$$

\textbf{Subcase 1.2: } $\lambda_1\geq\lambda_{n-\mathbf{p}+1}$.

First, it is easy to see that $\Lambda_{I_N}\leq \Lambda_I$ for any $I \in \mathfrak{J}(\mathbf{p},n)$. It implies $\Lambda | I_N \in \tilde{\Gamma}_k$.
Moreover, using $\Lambda_{I_N} \geq  -\frac{\epsilon}{2} \lambda_1$ and $-\lambda_n\geq \epsilon \lambda_1$, we obtain
$$\lambda_{n-\mathbf{p}+1}+ \lambda_{n-\mathbf{p}+2}+\cdots+\lambda_{n-1}\geq -\frac{\epsilon}{2} \lambda_1-\lambda_n\geq \frac{\epsilon}{2}\lambda_1,$$
 that means $\lambda_{n-\mathbf{p}+1}\geq \frac{\epsilon}{2(\mathbf{p}-1)}\lambda_1$. Further calculation yields
$$\Lambda_{I_N} \geq \frac{\epsilon}{2(\mathbf{p}-1)} \left(\lambda_1+\lambda_{n-\mathbf{p}+2}+\cdots+\lambda_n \right)=\frac{\epsilon}{2(\mathbf{p}-1)}\Lambda_{I'}.$$
For any $1\leq m \leq k-1$, we derive
\begin{equation*}
\begin{aligned}
\sigma_{m}(\Lambda |I_N)= &~\Lambda_{I'}\sigma_{m-1}(\Lambda |I'I_N )+\sigma_{m}(\Lambda | I' I_N)\\
\leq &~\frac{ 2(\mathbf{p}-1)}{\epsilon}\bigg[\Lambda_{I_N}\sigma_{m-1}(\Lambda |I'I_N )+\sigma_{m}(\Lambda | I' I_N)\bigg]\\
=&~\frac{ 2(\mathbf{p}-1)}{\epsilon}\sigma_{m}(\Lambda |I').
\end{aligned}
\end{equation*}
Then for any $I\in B$, by Proposition \ref{P4} (\romannumeral1) and $\Lambda | I_N \in \tilde{\Gamma}_k$,
\begin{align*}
\frac{\partial \left[\frac{\sigma_k(\Lambda(\lambda))}{\sigma_l(\Lambda(\lambda))}\right]}{\partial \Lambda_I}
&\leq \frac{\partial \left[\frac{\sigma_k(\Lambda(\lambda))}{\sigma_l(\Lambda(\lambda))}\right]}{\partial \Lambda_{I_N}} \\
&= \frac{\sigma_{k-1}(\Lambda |I_N)\sigma_{l}(\Lambda |I_N)-\sigma_{k}(\Lambda |I_N)\sigma_{l-1}(\Lambda |I_N)}{\sigma^2_l(\Lambda(\lambda))}  \\
&\leq \frac{\sigma_{k-1}(\Lambda |I_N)\sigma_{l}(\Lambda |I_N)}{\sigma^2_l(\Lambda(\lambda))}\\
&\leq \frac{ 4(\mathbf{p}-1)^2}{\epsilon^2} \frac{\sigma_{k-1}(\Lambda |I')\sigma_{l}(\Lambda |I')}{\sigma^2_l(\Lambda(\lambda))}\\
&\leq \frac{ 4(\mathbf{p}-1)^2}{\epsilon^2}  \left(1-\frac{l(N-k)}{k(N-l)}\right)^{-1} \frac{\sigma_{k-1}(\Lambda |I')\sigma_{l}(\Lambda |I')-\sigma_{k}(\Lambda |I')\sigma_{l-1}(\Lambda |I')}{\sigma^2_l(\Lambda(\lambda))} \\
&=\frac{ 4(\mathbf{p}-1)^2}{\epsilon^2}  \left(1-\frac{l(N-k)}{k(N-l)}\right)^{-1}\frac{\partial \left[\frac{\sigma_k(\Lambda(\lambda))}{\sigma_l(\Lambda(\lambda))}\right]}{\partial \Lambda_{I'}} .
\end{align*}
It follows that
$$\sum_{I\in A} \frac{\partial \left[\frac{\sigma_k(\Lambda(\lambda))}{\sigma_l(\Lambda(\lambda))}\right]}{\partial \Lambda_I}  \geq \frac{\epsilon^2}{ 4(\mathbf{p}-1)^2 C_{n-1}^{\mathbf{p}}} \left(1-\frac{l(N-k)}{k(N-l)}\right) \sum_{I\in B} \frac{\partial \left[\frac{\sigma_k(\Lambda(\lambda))}{\sigma_l(\Lambda(\lambda))}\right]}{\partial \Lambda_I} ,$$
and \eqref{xj-11} is proved.

\

\textbf{Case 2: } $\Lambda_{I_N} <  -\frac{\epsilon}{2} \lambda_1$.

From condition $\lambda_1 \geq \delta \lambda_2$, it follows that
\begin{equation*}\label{iuqourd-4}
\begin{aligned}
 \Lambda_{I_1} \leq (1+(\mathbf{p}-1)\delta^{-1})\lambda_1,
\quad \Lambda_{I_2}\leq \frac{\mathbf{p}}{\mathbf{p}-1+\delta}\Lambda_{I_1}.
\end{aligned}
\end{equation*}
Then for  $1\leq m \leq k-1$,
\begin{equation}\label{iuqourd-3}
\begin{aligned}
(-\Lambda_{I_N}) \sigma_{m-1} (\Lambda | I_1 I_N) &\geq \frac{\epsilon}{2}\frac{1}{1+(\mathbf{p}-1)\delta^{-1}} \Lambda_{I_1} \sigma_{m-1} (\Lambda | I_1 I_N).
\end{aligned}
\end{equation}

If there exists $s \in \{1, 2, \cdots, \mathbf{p} \}$ such that $\lambda_1\leq \lambda_s$ and $\lambda_1> \lambda_{s+1}$. Note that $\max \{\Lambda_I \mid I\in \mathfrak{J}(\mathbf{p}, n)\}=\Lambda_{I_1}$ with $I_1=(1, 2, \cdots,  \mathbf{p})$,
then from Proposition \ref{P1} (\romannumeral4),
\begin{equation}\label{iuqourd-1}
\Lambda_{I_1} \sigma_{m-1} (\Lambda |I_1 I_N) \geq \frac{m}{N-1}\sigma_{m}(\Lambda | I_N), \quad ~\forall~1\leq m \leq k-1.
\end{equation}

If $\lambda_{\mathbf{p}+1} \geq \lambda_1$, then $\max \{\Lambda_I \mid I\in \mathfrak{J}(\mathbf{p}, n)\}=\Lambda_{I_2}$. Moreover, for  $1\leq m \leq k-1$,
\begin{equation}\label{iuqourd-2}
\Lambda_{I_1} \sigma_{m-1} (\Lambda |I_1 I_N) \geq  \frac{\mathbf{p}-1+\delta}{\mathbf{p}} \Lambda_{I_2}   \sigma_{m-1}(\Lambda |I_2 I_N)\geq  \frac{(\mathbf{p}-1+\delta)m}{\mathbf{p}(N-1)}\sigma_{m}(\Lambda | I_N).
\end{equation}
Hence, from \eqref{iuqourd-3} to \eqref{iuqourd-2} and $\Lambda_{I_N} < 0$, it holds
\begin{align}
(-\Lambda_{I_N}) \sigma_{m-1} (\Lambda | I_1 I_N)
\geq \frac{\epsilon \delta m}{2(N-1)\mathbf{p}}\sigma_{m}(\Lambda | I_N)
\geq \theta \sigma_{m}(\Lambda)\label{iuqourd-33},
\end{align}
where $\theta=\frac{\epsilon \delta}{2(N-1)\mathbf{p}}$.

We $\mathbf{claim}$ that
\begin{equation}\label{i2ir34}
\sigma_{m} (\Lambda | I_1) \geq C_0 \sigma_{m} (\Lambda), \quad \forall~1\leq m \leq k-1.
\end{equation}
where  $C_0$ depending only on $n, k, l, \mathbf{p}, \delta$ and $\epsilon$.  In the following, we divide into two subcases to obtain \eqref{i2ir34}.

\textbf{Subcase 2.1: }  $\sigma_{m} (\Lambda | I_1) \geq \theta_1 (-\Lambda_{I_N}) \sigma_{m-1} (\Lambda | I_1 I_N)$, where $\theta_1=\frac{\epsilon \delta}{4(N-2)\mathbf{p}}$.

In this case, the result follows directly from  \eqref{iuqourd-33},
\begin{equation*}
\begin{aligned}
\sigma_{m} (\Lambda | I_1) \geq \theta_1 (-\Lambda_{I_N}) \sigma_{m-1} (\Lambda | I_1 I_N)\geq \theta_1 \theta \sigma_{m}(\Lambda).
\end{aligned}
\end{equation*}

\textbf{Subcase 2.2: }  $\sigma_{m} (\Lambda | I_1) < \theta_1 (-\Lambda_{I_N}) \sigma_{m-1} (\Lambda | I_1 I_N)$.

Note that
$$\Lambda_I\leq \max\{\Lambda_{I_1}, \Lambda_{I_2}\} \leq  \frac{\mathbf{p}}{\mathbf{p}-1+\delta}\Lambda_{I_1},~\quad \forall I \in \mathfrak{J}(\mathbf{p},n). $$
From \eqref{iuqourd-3} and \eqref{iuqourd-33}, it follows that
\begin{align*}
(m+1)\sigma_{m+1} (\Lambda | I_1) &= \sum_{s=2}^N \Lambda_{I_s} \sigma_{m}(\Lambda | I_1)- \sum_{s=2}^N \Lambda_{I_s}^2 \sigma_{m-1}(\Lambda | I_1 I_s)\\
&\leq\sum_{s=2}^N \Lambda_{I_s} \sigma_{m}(\Lambda | I_1)-  \Lambda_{I_N}^2 \sigma_{m-1}(\Lambda | I_1 I_N)\\
&\leq (N-2) \frac{\mathbf{p}}{\mathbf{p}-1+\delta}\Lambda_{I_1} \sigma_{m}(\Lambda | I_1)-  \Lambda_{I_N}^2 \sigma_{m-1}(\Lambda | I_1 I_N)\\
&\leq\frac{(N-2)\mathbf{p}\theta_1 \Lambda_{I_1} (-\Lambda_{I_N})}{\mathbf{p}-1+\delta} \sigma_{m-1} (\Lambda | I_1 I_N)+\frac{\epsilon\delta\Lambda_{I_N}\Lambda_{I_1}}{2(\mathbf{p}-1+\delta)} \sigma_{m-1} (\Lambda | I_1 I_N)\\
&\leq -\frac{\epsilon\delta}{4(\mathbf{p}-1+\delta)} (-\Lambda_{I_N})\Lambda_{I_1} \sigma_{m-1} (\Lambda | I_1 I_N)\\
&\leq -\frac{\epsilon\delta}{4(\mathbf{p}-1+\delta)} \theta \Lambda_{I_1} \sigma_{m} (\Lambda).
\end{align*}
Then
\begin{equation*}
\begin{aligned}
\sigma_{m} (\Lambda | I_1)=\frac{\sigma_{m+1}(\Lambda)-\sigma_{m+1}( \Lambda | I_1)}{\Lambda_{I_1}}
\geq -\frac{\sigma_{m+1}( \Lambda | I_1)}{\Lambda_{I_1}}
\geq C_m \sigma_{m} (\Lambda),
\end{aligned}
\end{equation*}
where $C_m:=\frac{\epsilon\delta}{4(m+1)(\mathbf{p}-1+\delta)} \theta$.
 The \textbf {Claim \eqref{i2ir34}} holds.

 By the Newton-MacLaurin inequality, we obtain
\begin{align}
\frac{\partial \left[\frac{\sigma_k(\Lambda(\lambda))}{\sigma_l(\Lambda(\lambda))}\right]}{\partial \Lambda_{I_1}}
&= \frac{\sigma_{k-1}(\Lambda |I_1)\sigma_{l}(\Lambda |I_1)-\sigma_{k}(\Lambda |I_1)\sigma_{l-1}(\Lambda |I_1)}{\sigma^2_l(\Lambda(\lambda))}\notag\\
&\geq \left(1-\frac{l(N-k)}{k(N-l)}\right)\frac{\sigma_{k-1}(\Lambda |I_1)\sigma_{l}(\Lambda |I_1)}{\sigma^2_l(\Lambda(\lambda))}\notag\\
&\geq  \frac{N(k-l)}{k(N-l)} C_{k-1}C_l \frac{\sigma_{k-1}(\Lambda (\lambda))}{\sigma_l(\Lambda(\lambda))}\label{iuqourd-5533}.
\end{align}
On the other hand,
\begin{align*}
 \sum_I\frac{\partial \left[\frac{\sigma_k(\Lambda(\lambda))}{\sigma_l(\Lambda(\lambda))}\right]}{\partial \Lambda_{I}}
 &=\sum_I\frac{\sigma_{k-1}(\Lambda|I)\sigma_l(\Lambda)-\sigma_k(\Lambda)\sigma_{l-1}(\Lambda|I)}{\sigma_l^2(\Lambda)}\\
 &=\frac{(N-k+1)\sigma_{k-1}(\Lambda)\sigma_l(\Lambda)-(N-l+1)\sigma_k(\Lambda)\sigma_{l-1}(\Lambda)}{\sigma_l^2(\Lambda)}\\
&\leq(N-k+1) \frac{\sigma_{k-1}(\Lambda)}{\sigma_l(\Lambda)}.
\end{align*}
Combining this with \eqref{iuqourd-5533}, we deduce that
\begin{align*}
\frac{\partial \left[\frac{\sigma_k(\Lambda(\lambda))}{\sigma_l(\Lambda(\lambda))}\right]}{\partial \lambda_1}
\geq\frac{\partial \left[\frac{\sigma_k(\Lambda(\lambda))}{\sigma_l(\Lambda(\lambda))}\right]}{\partial \Lambda_{I_1}}
&\geq  \frac{N(k-l)}{k(N-l)} \frac{C_{k-1}C_l}{(N-k+1)} \sum_I \frac{\partial \left[\frac{\sigma_k(\Lambda(\lambda))}{\sigma_l(\Lambda(\lambda))}\right]}{\partial \Lambda_I} \\
&=\frac{N(k-l)}{k(N-l)}\frac{C_{k-1}C_l}{(N-k+1)\mathbf{p}} \sum_i\frac{\partial \left[\frac{\sigma_k(\Lambda(\lambda))}{\sigma_l(\Lambda(\lambda))}\right]}{\partial \lambda_i}.
\end{align*}
\end{proof}

 For convenience, we make the following notations
$$
F(D^2 u):=\frac{\sigma_k(\Lambda(D^2 u))}{\sigma_l(\Lambda(D^2 u))}, \quad F^{ij}=\frac{\partial F}{\partial u_{ij}}, \quad \quad~F^{ij,kl}=\frac{\partial^2 F}{\partial u_{ij} u_{kl}}.$$
$$
\tilde F(D^2 u):=\bigg[\frac{\sigma_k(\Lambda(D^2 u))}{\sigma_l(\Lambda(D^2 u))}\bigg]^{\frac{1}{k-l}}, \quad \tilde{F}^{ij}=\frac{\partial \tilde{F}}{\partial u_{ij}}, \quad \quad~\tilde{F}^{ij,kl}=\frac{\partial^2 \tilde{F}}{\partial u_{ij} u_{kl}}.$$
Therefore, equation \eqref{(1.1)} can be written by
\begin{eqnarray*}\label{(2.11)}
\tilde{F}(D^2 u)= f^\frac{1}{k-l}(x).
\end{eqnarray*}
Lastly, we list the following well-known result.

\begin{lemma}\label{P7}
If $W=(w_{ij})$ is a symmetric real matrix, $\lambda_i=\lambda_i(W)$
is one of the eigenvalues $(i = 1, ..., n)$ and
$F=F(W)=f(\lambda(W))$ is a symmetric function of $\lambda_1, ...,
\lambda_n$, then for any real symmetric matrix $A= (a_{ij})$, we
have the following formulas:
\begin{eqnarray*}\label{Pre-2}
\frac{\partial^2 F}{\partial w_{ij}\partial
w_{st}}a_{ij}a_{st}=\frac{\partial^2 f}{\partial\lambda_p
\partial\lambda_q}a_{pp}a_{qq}+2\sum_{p<q}\frac{\frac{\partial
f}{\partial \lambda_p}-\frac{\partial f}{\partial
\lambda_q}}{\lambda_p-\lambda_q}a^{2}_{pq}.
\end{eqnarray*}
\end{lemma}
\begin{proof}
The proof can be found in \cite{CG}.
\end{proof}

\section{$C^0$ estimates}
In this section, we establish the $C^0$ estimates for the $(\Lambda, k)$-admissible solution of equation \eqref{(1.7)}.

\begin{theorem}\label{C0}
Let $\Omega\subset\mathbb{R}^{n}$  be a bounded domain with $C^1$ boundary, $f$ be a positive function  and $\varphi\in C^{0}(\partial\Omega)$. Suppose that $u\in C^{2}(\Omega)\cap C^{1}(\overline{\Omega})$ is a $(\Lambda, k)$-admissible solution of \eqref{(1.7)} in $\Omega$ with $\varepsilon\in(0,1)$. Then
\begin{eqnarray*}\label{(3.1)}
\max_{\overline{\Omega}}|\varepsilon u|\leq M_{0},
\end{eqnarray*}
where $M_{0}$ depends on $n, k, l , \mathbf{p}, \sup_{{\Omega}}f, |\varphi|_{C^{0}}$ and $\Omega.$
\end{theorem}

\begin{proof}
On the one hand, note that $\Delta u>0$ since $\Lambda(D^2 u)\in \tilde{\Gamma}_k\subset \tilde{\Gamma}_1$. Hence $u$ attains its maximum value at $x_0 \in \partial \Omega $, then
$$0\leq u_\nu(x_0)=-\varepsilon u(x_0)+\varphi(x_0).$$
Therefore,
\begin{eqnarray*}\label{(3.2)}
\max_{{\overline\Omega}}\varepsilon u\leq\sup_{\partial\Omega} |\varphi|.
\end{eqnarray*}

On the other hand, suppose that $0\in\Omega$ and  take
\begin{eqnarray*}
v=A|x|^{2},
\end{eqnarray*}
where $A=\frac{1}{2\mathbf{p}}\left[\frac{C_{N}^{l}}{C_{N}^{k}}\sup_{{\Omega}}f \right]^{\frac{1}{k-l}}$. It is easy to see that
\begin{eqnarray*}
F(D^2u)\leq\sup_{{\Omega}}f\leq F(D^2v).
\end{eqnarray*}
By the comparison principle, $u-v$ attains its minimum point at $x_{1} \in \partial \Omega$, then
\begin{eqnarray*}
0\geq(u-v)_\nu(x_1)=-\varepsilon u(x_1)+\varphi(x_1)-2Ax_1\cdot\nu.
\end{eqnarray*}
For all $x\in\overline{\Omega}$,
\begin{eqnarray*}
\varepsilon(u-v)(x)\geq\varepsilon(u-v)(x_1)\geq-2A\mathrm{diam}(\Omega)-|\varphi|_{C^0}-A(\mathrm{diam}(\Omega))^2.
\end{eqnarray*}
Thus,
\begin{eqnarray*}\label{(3.3)}
\min_{\bar\Omega}\varepsilon u\geq\min_{\bar\Omega}\varepsilon(u-v)\geq-|\varphi|_{C^{0}}-A(\mathrm{diam}(\Omega))^{2}-2A\mathrm{diam}(\Omega).
\end{eqnarray*}
\end{proof}

\begin{remark}
It is worth noting that the assumption $\Lambda(D^{2}u)\in\tilde\Gamma_{k}$ implies that $f$ does not vanish in $\Omega$, and the estimate on $\left|\varepsilon u\right|_{0,\overline{\Omega}}$ is independent of the lower bound of $f$. If $f\geq 0$, we can obtain the same estimate by considering $f_{\delta}=f+\delta$ first, then let $\delta\rightarrow0$ to complete the proof. We will not repeat this argument in subsequent sections where it applies.
\end{remark}

\section{global gradient estimates}
Following a similar argument of the complex Monge-Amp\'ere equation in  Li \cite {LSY}, we consider the global gradient  estimates (independent of $\varepsilon$) for the $(\Lambda, k)$-admissible solution of equation \eqref{(1.7)}  in this section.

\begin{theorem}\label{C1}
Let $\Omega\subset\mathbb{R}^{n}$ be a $C^3$ strictly convex domain,  $f$ be a positive function with $f^{\frac{1}{k-l}}\in C^{1}(\overline{\Omega})$  and  $\varphi\in C^{2}(\partial\Omega)$. Suppose that $u\in C^{3}(\Omega)\cap C^{2}(\overline{\Omega})$ is a $(\Lambda, k)$-admissible solution of \eqref{(1.7)}, then
\begin{eqnarray*}\label{(4.1)}
\sup_{\overline{\Omega}}|Du|\leq M_{1},
\end{eqnarray*}
where $M_{1}$ depends on $n, k, l, \mathbf{p}, M_{0}, |\varphi|_{C^2}, |f^{\frac1{k-1}}|_{C^1}$ and $\Omega$.
\end{theorem}
\begin{proof}
It suffices to prove that
\begin{eqnarray*}\label{(4.2)}
D_\xi u(x)\leq M_1,\quad\forall\left(x,\xi\right)\in\overline{\Omega}\times\mathbb{S}^{n-1}.
\end{eqnarray*}

For any $(x,\xi)\in\overline{\Omega}\times\mathbb{S}^{n-1}$, consider the following test function
\begin{eqnarray*}\label{(4.3)}
G(x,\xi)=D_{\xi}u(x)-\langle\nu,\xi\rangle(-\varepsilon u+\varphi(x))+\varepsilon^{2}u^{2}+K|x|^{2},
\end{eqnarray*}
where $K$ is a large constant to be determined later, and $\nu$ is a $C^{2}(\overline{\Omega})$ extension of the outer unit normal vector field on $\partial\Omega$.
If one can prove that
\begin{eqnarray}\label{(4.4)}
G(x,\xi)\leq C,    ~\quad ~ \forall (x,\xi)\in\overline{\Omega}\times\mathbb{S}^{n-1},
\end{eqnarray}
then \eqref{(4.2)} follows immediately.

Assume $G$ achieves its maximum at $(x_0, \xi_0)$ $\in$ $\overline {\Omega }\times$ $\mathbb{S} ^{n-1}$, then $G(x,\xi_{0})$ achieves its maximum at the same point $x_{0}\in\bar{\Omega}$.
Clearly, $D_{\xi_{0}}u(x_{0})>0$; otherwise, the result follows immediately. We distinguish two cases based on $x_0$.

\textbf{Case 1: }$x_0\in \Omega.$

Firstly,  rotate the coordinates so that $D^{2}u(x_{0})$ is diagonal.
A direct calculation at  $x_0$ gives that
\begin{equation*}
0=G_i=u_{i\xi_{0}}-\langle v,\xi_0\rangle_i(-\varepsilon u+\varphi)-\langle v,\xi_0\rangle(-\varepsilon u_{i}+\varphi_{i})+2\varepsilon^{2}uu_{i}+2Kx_{i}.
\end{equation*}
Then combining with Proposition \ref{P4} (\romannumeral3),
\begin{align*}
0\geq& \widetilde{F}^{ii}G_{ii}\\
=&(f^{\frac{1}{k-l}})_{\xi_0}+\widetilde{F}^{ii}u_{ii}[\varepsilon\langle\nu,\xi_{0}\rangle+2\varepsilon^{2}u]
+\widetilde{F}^{ii}[2\varepsilon^{2}u_{i}^{2}+2\langle\nu,\xi_{0}\rangle_{i}\varepsilon u_{i}]\\
&+\widetilde{F}^{ii}[2K-\langle\nu,\xi_0\rangle_{ii}(-\varepsilon u+\varphi)-\langle\nu,\xi_0\rangle\varphi_{ii}-2\langle\nu,\xi_0\rangle_i\varphi_i]\\
\geq&-|f^{\frac{1}{k-l}}|_{C^1}-|f^{\frac{1}{k-l}}|_{C^0}[1+2M_0]\\
&+\sum_i\widetilde{F}^{ii}\left[2K-|D\langle\nu,\xi_0\rangle|^2-|D^2\langle\nu,\xi_0\rangle|(M_0+|\varphi|_{C^0}) -|D^2\varphi|-2|D\langle\nu,\xi_0\rangle||D\varphi|\right]\\
>&0,
\end{align*}
if we choose $K$ large enough, which depending only on $n, k, l, \mathbf{p}, M_{0}, |\varphi|_{C^2}, |f^{\frac1{k-1}}|_{C^1}$ and $\Omega$,  this is a contradiction.

\textbf{Case 2: }$x_0\in \partial \Omega$. We consider two subcases for $\xi_{0}$.

\textbf{Subcase 2.1: }$\xi_{0}$ is tangential at $x_{0}\in\partial\Omega$. Without loss of generality, the outer normal direction of $\Omega$ at  the boundary point $x_0$ is $(0,\ldots,0,1)$. By a rotation,  assume that $\xi_{0}=(1,\ldots,0)=e_{1}$ and $\xi(t)={\frac{(1,t,0,\ldots,0)}{\sqrt{1+t^{2}}}}$, then
\begin{equation*}
0=\left.\frac{dG(x_0,\xi(t))}{dt}\right|_{t=0}=u_2(x_0)-\nu^2(-\varepsilon u+\varphi(x_0)),
\end{equation*}
where $\nu^2$ denotes the second component of $\nu$. So
\begin{equation*}
|u_2|(x_0)\leq C,
\end{equation*}
where the constant $C$ depends on $|\varphi|_{C_0}$, $M_{0}$ and $\Omega$.
Similarly,  $\left|u_{i}\right|(x_0)\leq C$  holds for any $ i>1$. By the Hopf Lemma,
\begin{equation*}
\begin{aligned}
0\leq D_{\nu}G(x_{0},\xi_{0})& =D_{\nu}D_{1}u-D_{\nu}[\langle\nu,\xi_{0}\rangle(-\varepsilon u+\varphi)]+2\varepsilon^2 u D_{\nu}u+2Kx_0\cdot \nu \\
&\leq D_{\nu}D_{1}u+C \\
&=D_{1}D_{\nu}u-D_{1}\nu_{k}D_{k}u+C.
\end{aligned}
\end{equation*}
From the boundary condition,
\begin{equation*}
D_1D_\nu u=D_1(-\varepsilon u+\varphi)\leq D_1\varphi.
\end{equation*}
Following the argument of \cite {LSY}, we derived that
\begin{equation*}
-D_1\nu_kD_ku\leq-\kappa_\text{min}G(x_0,\xi_0)+C,
\end{equation*}
where $\kappa_{\mathrm{min}}$ is the minimum principal curvature of $\partial\Omega$. Therefore,
\begin{eqnarray*}
G(x_0,\xi_0)\leq\frac{|D\varphi|+C}{\kappa_{\min}}.
\end{eqnarray*}

\textbf{Subcase 2.2: }$\xi_{0}$ is non-tangential at $x_0\in\partial\Omega$.  Write $\xi_{0}=\alpha\tau+\beta\nu$, where $\tau\in\mathbb{S}^{n-1}$ is tangential at $x_0$, that is $\langle\tau,\nu\rangle=0$, $\alpha=\langle\xi_{0},\tau\rangle\in[0,1)$,  and $\alpha^{2}+\beta^{2}=1$. A straightforward calculation shows that
 \begin{equation*}
G(x_{0},\tau) = D_{\tau}u+\varepsilon^{2}u^{2}+K|x_{0}|^{2},
\end{equation*}
and
\begin{equation*}
\begin{aligned}
G(x_{0},\xi_{0})& =\alpha D_{\tau}u+\varepsilon^{2}u^{2}+K|x_{0}|^{2} \\
&\leq\alpha G(x_{0},\xi_{0})+(1-\alpha)(\varepsilon^{2}u^{2}+K|x_{0}|^{2}).
\end{aligned}
\end{equation*}
It implies
\begin{equation*}
G(x_0,\xi_0)\leq\varepsilon^2u^2+K|x_0|^2\leq C.
\end{equation*}
In summary, \eqref{(4.4)} holds, and the proof is complete.

\end{proof}


\section{global estimates for second-order derivatives}
In this section, we prove the global second-order  derivatives estimates for the $(\Lambda, k)$-admissible solution of equation \eqref{(1.7)} by Lemma \ref{f11} and  \ref{l2}.

\subsection{Reduce global  second-order derivatives to double normal  second-order derivatives on the boundary}
\begin{theorem}\label{C211}
Let $\Omega\subset\mathbb{R}^{n}$ be a bounded and strictly convex domain with $C^4$ boundary, $\varphi\in C^3(\partial\Omega)$ and $f$ be a positive function with $f^{\frac1{k-l}}\in C^2(\overline{\Omega})$. Suppose that $u\in C^{4}(\Omega)\cap C^{3}(\overline{\Omega})$ is a $(\Lambda,k)$-admissible solution of \eqref{(1.7)} in $\Omega$, then
\begin{equation*}\label{(5.11)}
\sup_{\overline{\Omega}}|D^2u|\leq M_2(1+\sup_{\partial\Omega}\left|u_{\nu\nu}\right|),
\end{equation*}
where $M_2$ depends on
$n, k, l, \mathbf{p}, M_{0}, M_{1}, |f^{\frac{1}{k-l}}|_{C^{2}},|\varphi|_{C^{3}}$ and $\Omega$.
\end{theorem}

\begin{proof}
It suffices to prove that
\begin{equation*}\label{(5.2)}
u_{\xi\xi}\leq M_2(1+\sup_{\partial\Omega}\left|u_{\nu\nu}\right|),\quad\forall\xi\in\mathbb{S}^{n-1}.
\end{equation*}

Assume $0\in \Omega$ and consider the following auxiliary function in $\bar{\Omega}\times\mathbb{S}^{n-1}$:
\begin{equation*}\label{(5.13)}
H(x,\xi):=u_{\xi\xi}-v(x,\xi)+B|x|^2+|Du|^2,
\end{equation*}
where $v(x,\xi)=a^kD_ku+b:=2(\xi\cdot\nu)\xi'\cdot(D(-\varepsilon u+\varphi(x))-D_kuD\nu^k)$ with $\xi^{\prime}=\xi-(\xi\cdot\nu)\nu$, $a^{k} = 2( \xi \cdot \nu ) ( - \varepsilon \xi ^{\prime k} - \xi ^{\prime i}D_{i}\nu ^{k}), ~b = 2( \xi \cdot \nu ) \xi ^{\prime k}\varphi _{x_{k}}$, $B$ is a positive constant to be determined  later.

Assume $H$ achieves its maximum at $(x_0, \xi_0)$ $\in$ $\overline {\Omega }\times$ $\mathbb{S} ^{n-1}$,  then  $H(x,\xi_{0})$ achieves its maximum at the same point $x_{0}\in\bar{\Omega}$.  Rotate the coordinates so that $D^2u(x_0)$ is diagonal and consider   two cases for $x_0$.

\textbf{Case 1: }$x_0\in \Omega.$  By a direct computation,
\begin{equation*}\label{(5.14)}
0=H_{i}=u_{\xi\xi i}-a_i^{k}u_{k}-a^{k}u_{ki}-b_i+2Bx_{i}+2u_{k}u_{ki},
\end{equation*}
and
\begin{equation}\label{(5.15)}
\begin{aligned}
0\geq H_{ij}=&u_{\xi\xi ij}-a_{ij}^ku_k-a_i^ku_{kj}-a_j^ku_{ki}-a^ku_{kij}-b_{ij}\\
&+2B\delta_{ij}+2u_{ki}u_{kj}+2u_ku_{kij}.
\end{aligned}
\end{equation}
According to Lemma  \ref{P7},
\begin{equation}\label{(5.16)}
\widetilde{F}^{ij}u_{ij\xi\xi}=(f^{\frac{1}{k-l}})_{\xi\xi}-\widetilde{F}^{ij,kl}u_{ij\xi}u_{kl\xi}
\geq(f^{\frac{1}{k-l}})_{\xi\xi}.
\end{equation}
Now  contract \eqref{(5.15)} with $\tilde{F}^{ij}$ and use \eqref{(5.16)}, we obtain
\begin{align*}
0\geq\widetilde{F}^{ij}H_{ij}=&\widetilde{F}^{ij}u_{ij\xi\xi}-\widetilde{F}^{ij}a_{ij}^{k}u_{k}
-2\widetilde{F}^{ij}a_{i}^{k}u_{kj}-\widetilde{F}^{ij}a^{k}u_{ijk}-\widetilde{F}^{ij}b_{ij}\notag\\
&+2B\sum_i\tilde{F}^{ii}+2\widetilde{F}^{ij}u_{ik}u_{jk}
+2\widetilde{F}^{ij}u_{ijk}u_{k}\notag\\
\geq& -C(1+\sum_i\tilde{F}^{ii})+2B\sum_i\tilde{F}^{ii}
+2\widetilde{F}^{ii}u_{ii}^{2}-C\widetilde{F}^{ii}|u_{ii}|\label{(5.17)},
\end{align*}
where $C$ is a positive constant depending on
 $M_1,  |\varphi|_{C^{3}}, |f^{\frac{1}{k-1}}|_{C^{2}}$ and $\Omega$.
Combining with Proposition \ref{P4} (\romannumeral3), it is easy to get
$$0\geq\widetilde{F}^{ij}H_{ij}\geq2\widetilde{F}^{ii}\left(u_{ii}-\frac{C}{4}\right)^{2}
+\left(2B-\frac{C^{2}}{8}-C\right){\sum_i\tilde{F}^{ii}}-C>0,$$
if we choose $B>\frac{C^{2}}{8}+C+\frac{C}{ C_\mathbf{p}}$, where $C_\mathbf{p}$ is a constant defined in Proposition \ref{P4} (\romannumeral3). This is a contradiction.

\textbf{Case 2: }$x_0\in \partial \Omega$. We consider two subcases for $\xi_{0}$.

\textbf{Subcase 2.1: }$\xi_{0}$ is tangential at $x_{0}\in\partial\Omega$. We directly have $\xi_{0}\cdot\nu=0$,  $\nu=-Dd$ with $d(x) = \text{dist}(x, \partial\Omega)$, $v(x_{0},\xi_{0})=0$. Without loss of generality, we may assume that  $u_{\xi_{0}\xi_{0}}(x_{0})>0$; otherwise, the proof is complete. Using the method of \cite {MQ}, define
\begin{equation*}\label{(5.18)}
c^{ij}=\delta_{ij}-\nu^{i}\nu^{j}\quad\mathrm{in}\: \Omega,
\end{equation*}
clearly  $c^{ij}D_{j}$ is a tangential direction on $\partial\Omega$.
From the boundary condition,
\begin{equation*}\label{(5.619)}
u_{kj}\nu^{k}+u_{k}D_{j}\nu^{k}=-\varepsilon u_{j}+\varphi_j.
\end{equation*}
Multiplying both sides by $c^{ij}$ gives
\begin{equation}\label{(5.19)}
u_{ki}\nu^{k}=-\varepsilon c^{ij}u_{j}+c^{ij}\varphi_j-c^{ij}u_{k}D_{j}\nu^{k}+\nu^{i}\nu^{j}\nu^{k}u_{kj}.
\end{equation}
Differentiating \eqref{(5.19)} once more yields
\begin{equation}\label{(5.110)}
u_{kiq}\nu^{k}+u_{ki}D_q\nu^k
=D_{q}(-\varepsilon c^{ij}u_{j}+c^{ij}\varphi_{j}-c^{ij}u_{k}D_{j}\nu^{k}+\nu^{i}\nu^{j}\nu^{k}u_{kj}),
\end{equation}
Then, multiplying \eqref{(5.110)} by $c^{pq}$, we obtain
\begin{equation*}
\begin{aligned}
u_{kip}\nu^{k}
=&c^{pq}D_{q}(-\varepsilon c^{ij}u_{j}+c^{ij}\varphi_{j}-c^{ij}u_{k}D_{j}\nu^{k}+\nu^{i}\nu^{j}\nu^{k}u_{kj})
-c^{pq}u_{ki}D_{q}\nu^{k}+\nu^{p}\nu^{q}\nu^{k}u_{kiq}\\
=&D_{p}(-\varepsilon c^{ij}u_{j}+c^{ij}\varphi_{j}-c^{ij}u_{k}D_{j}\nu^{k}+\nu^{i}\nu^{j}\nu^{k}u_{kj})-u_{ki}D_{p}\nu^{k}.
\end{aligned}
\end{equation*}
Hence,
\begin{equation}\label{(5.112)}
\begin{aligned}
u_{\xi_{0}\xi_{0}\nu} &=\sum_{i,k,p=1}^{n}\xi_{0}^{i}\xi_{0}^{p}u_{kip}\nu^{k} \\
&=\sum_{i,k,p=1}^{n}\xi_{0}^{i}\xi_{0}^{p}[D_{p}(-\varepsilon c^{ij}u_{j}+c^{ij}\varphi_{j}-c^{ij}u_{k}D_{j}\nu^{k}+\nu^{i}\nu^{j}\nu^{k}u_{kj})-u_{ki}D_{p}\nu^{k}] \\
&\leq-2\xi_{0}^{i}\xi_{0}^{p}u_{ki}D_{p}\nu^{k}+C(1+|u_{\nu\nu}|).
\end{aligned}
\end{equation}
Note that  the assumption $u_{\xi_{0}\xi_{0}}(x_{0})>0$ is used in the last line.

After a suitable rotation, we may take $\xi_{0}=(1,\ldots,0)=e_{1}$ and $\xi(t)={\frac{(1,t,0,\ldots,0)}{\sqrt{1+t^{2}}}}$, then
\begin{equation*}\label{(5.113)}
0=\left.\frac{dH(x_0,\xi(t))}{dt}\right|_{t=0}=2u_{12}(x_0)-2\nu^2(-\varepsilon u_1+\varphi_1-D_kuD\nu^k).
\end{equation*}
Therefore,
\begin{equation*}
|u_{12}|(x_0)\leq C.
\end{equation*}
Similarly, $|u_{1i}|(x_0)\leq C$ holds for any $i>1$. According to \eqref{5.552} in Lemma \ref{c26}, we get $D_{1}\nu^{1}>\kappa_{\min}$, since $\Omega$  is strictly convex.
Then \eqref{(5.112)} becomes to
\begin{equation}\label{(5.115)}
u_{\xi_{0}\xi_{0}\nu}\leq-2\kappa_{\min}u_{11}+C(1+|u_{\nu\nu}|).
\end{equation}
Moreover,  from the Hopf lemma  and \eqref{(5.115)},
\begin{equation*}
\begin{aligned}
0 &\leq H_{\nu}(x_{0},\xi_{0}) \\
&=u_{\xi_{0}\xi_{0}\nu}-D_{\nu}a^{k}u_{k}-a^{k}u_{k\nu}-D_{\nu}b+2Bx_{k}\nu^{k}+2u_{k}u_{k\nu} \\
&\leq-2\kappa_{\mathrm{min}}u_{\xi_{0}\xi_{0}}+C(1+|u_{\nu\nu}|).
\end{aligned}
\end{equation*}
It follows that
\begin{equation*}
u_{\xi_0\xi_0}(x_0)\leq C(1+|u_{\nu\nu}|).
\end{equation*}

\textbf{Subcase 2.2: } $\xi_{0}$ is non-tangential.
Write $\xi_{0}=\alpha\tau+\beta\nu$, where $\tau\in\mathbb{S}^{n-1}$ is tangential at $x_0$, that is $\langle\tau,\nu\rangle=0$, $\alpha=\langle\xi_{0},\tau\rangle>0$, $\beta=\langle\xi_{0},\nu\rangle<1$ and $\alpha^{2}+\beta^{2}=1$. Combining the boundary condition and the definition of $v(x_{0},\xi_{0})$, we obtain
\begin{equation*}
\begin{aligned}
u_{\xi_{0}\xi_{0}}(x_{0})&=\alpha^{2}u_{\tau\tau}(x_{0})+\beta^{2}u_{\nu\nu}(x_{0})+2\alpha\beta u_{\tau\nu}(x_{0})\\
&=\alpha^{2}u_{\tau\tau}(x_{0})+\beta^{2}u_{\nu\nu}(x_{0})+2(\xi_{0}\cdot\nu)\xi_{0}^{\prime}\cdot(D\varphi-\varepsilon Du-u_{l}D\nu^{l})\\
&=\alpha^{2}u_{\tau\tau}(x_{0})+\beta^{2}u_{\nu\nu}(x_{0})+v(x_{0},\xi_{0}).
\end{aligned}
\end{equation*}
Therefore,
\begin{equation*}
\begin{aligned}
H(x_{0},\xi_{0})& =\alpha^{2}H(x_{0},\tau)+\beta^{2}H(x_{0},\nu) \\
&\leq\alpha^{2}H(x_{0},\xi_{0})+\beta^{2}H(x_{0},\nu).
\end{aligned}
\end{equation*}
Further calculation yields
\begin{equation*}
H(x_0,\xi_0)\leq H(x_0,\nu)=u_{\nu\nu}+B|x|^2+|Du|^2.
\end{equation*}
It implies that
\begin{equation*}
u_{\xi_0\xi_0}(x_0)\leq|u_{\nu\nu}|+C.
\end{equation*}
In conclusion, the proof is complete.
\end{proof}

\begin{theorem}\label{C22}
Let~$u\in C^4(\Omega) \cap C^3(\overline\Omega)$~be the $(\Lambda,k)$-admissible solution of equation ~\eqref{(1.7)}. Denoting the tangential direction $\tau$ and the outer unit normal $\nu$ at any point $y \in \partial \Omega$,
then
$$
|D_{\tau\nu}u(y)|\leq C,
$$
where $C$ is a positive constant depending on $n, k, l, \mathbf{p}, |u|_{C^1}, |\varphi|_{C^1}$ and $\Omega$.
\end{theorem}
\begin{proof}
The idea of the proof is similar to the proof of Lemma 12 in \cite{MQ}.
\end{proof}

\subsection{ Double normal estimates on the boundary}\

\

In this section, we  prove $\underset{\partial\Omega}{\max}|D_{\nu\nu}u|\leq C$, this requires considering two cases:
$\underset{\partial\Omega}{\max}|u_{\nu\nu}|= \underset{\partial\Omega}{\max}~u_{\nu\nu}$ and
$\underset{\partial\Omega}{\max}|u_{\nu\nu}|= -\underset{\partial\Omega}{\min}~u_{\nu\nu}$. Firstly, we denote
$$h(x) = -d(x) + d^2(x),$$
where $d(x) = \text{dist}(x, \partial\Omega)$ is the distance function of $\Omega$. For a sufficiently small constant \(\mu>0\), define
$$\Omega_\mu = \{x \in \Omega :0< d(x) < \mu\},$$
then we  give the following Lemma.
\begin{lemma}\label{c26}
Suppose $\Omega \subset \mathbb{R}^n$ is a $C^2$  strictly convex domain, and $u \in C^2(\Omega)$ is the $(\Lambda, k)$-admissible solution of Hessian quotient equation ~\eqref{(1.7)}.
Then
$$\sum_{i} \tilde{F}^{ii} h_{ii} \ge c_0 \left(\sum_{i} \tilde{F}^{ii} + 1\right), \quad \text{in} \ \Omega_\mu, $$
where $\mu \in (0, \frac{1}{10})$  depending only on $\Omega$ and $c_0$  depending  on $n, k, l, \mathbf{p}, \Omega$.

\end{lemma}

\begin{proof}
From \cite{GT},  $d$ is $C^4$ in $\Omega_\mu $ and it holds
$$|Dd| = 1 \quad \text{in } \bar{\Omega}_\mu; \quad -Dd = \nu \quad \text{on } \partial\Omega_\mu.$$
For any $x_0 \in \Omega_\mu$, there is a $y_0 \in \partial\Omega$ such that $d(x_0) = |x_0 - y_0|$. Moreover, in the principal coordinate system,
$$
-Dd(x_0) = \nu(y_0) = (0, \dots, 0, 1),$$
$$-D^2d(x_0) = \text{diag}\left\{\frac{\kappa_1(y_0)}{1 - \kappa_1(y_0)d(x_0)}, \dots, \frac{\kappa_{n-1}(y_0)}{1 - \kappa_{n-1}(y_0)d(x_0)}, 0\right\},$$
where $\kappa_1(y_0), \dots, \kappa_{n-1}(y_0)$ are the principal curvature of $\partial\Omega$ at $y_0$. Since $\Omega$ is strictly convex, then there exists two positive constants $\kappa_{\min} < 1$ and $\kappa_{\max}$ depending only on $\Omega$ and $\mu$ such that
\begin{equation}\label{5.552}
\begin{aligned}
\kappa_{\min}\text{diag}\{1, \dots, 1, 0\} \leq -D^2d(x_0) \leq \kappa_{\max}\text{diag}\{1, \dots, 1, 0\},
 \end{aligned}
\end{equation}
in the principal coordinate system. Hence
\begin{equation}\label{5.53}
\begin{aligned}
\kappa_{\min} \operatorname{diag}\{1, \ldots, 1, 1\} \leq D^{2} h(x_{0}) \leq\left(\kappa_{\max}+1\right) \operatorname{diag}\{1, \ldots, 1, 1\}
 \end{aligned}
\end{equation}
in the principal coordinate system.

If $D^{2} u(x_{0})$ is diagonal, and denote $\lambda=(\lambda_{1}, \ldots, \lambda_{n})$ with $\lambda_{i}=u_{ii}$. Assume $\lambda_{1} \geq \lambda_{2} \geq \cdots \geq \lambda_{n}$, then $\Lambda_{I_1}\geq \cdots\geq \Lambda_{I_N}$. By a direct computation (details can refer to \cite{ZJ}), we obtain
\begin{eqnarray}\label{5.51}
\begin{aligned}
\tilde{F}^{ii}=&\frac{1}{k-l}\left[\frac{\sigma_k(\Lambda)}{\sigma_l(\Lambda)}\right]^{\frac{1}{k-l}-1}\sum_{\{I|i\in I\}}\frac{\sigma_{k-1}(\Lambda|\Lambda_{I})\sigma_{l}(\Lambda)-\sigma_{k}(\Lambda)\sigma_{l-1}(\Lambda|\Lambda_{I})}{\sigma_{l}^{2}(\Lambda)} \\
\geq&\frac{(N+1)}{k(N-l+1)}\left[\frac{\sigma_k(\Lambda)}{\sigma_l(\Lambda)}\right]^{\frac{1}{k-l}-1}\sum_{\{I|i\in I\}}\frac{\sigma_{k-1}(\Lambda|\Lambda_{I})\sigma_{l}(\Lambda|\Lambda_{I})}{\sigma_{l}^{2}(\Lambda)},
 \end{aligned}
\end{eqnarray}
and
\begin{equation*}\label{5.52}
\begin{aligned}
 \sum_{i}\tilde{F}^{ii}
& = \frac{1}{k-l}\left[\frac{\sigma_{k}(\Lambda)}{\sigma_{l}(\Lambda)}\right]^{\frac{1}{k-l}-1} \sum_{i=1}^{n} \sum_{\{I|i\in I\}} \frac{\sigma_{k-1}(\Lambda|\Lambda_{I}) \sigma_{l}(\Lambda) - \sigma_{k}(\Lambda) \sigma_{l-1}(\Lambda|\Lambda_{I})}{\sigma_{l}^{2}(\Lambda)} \\
& \leq \frac{\mathbf{p}(N-k+1)}{k-l}\left[\frac{\sigma_{k}(\Lambda)}{\sigma_{l}(\Lambda)}\right]^{\frac{1}{k-l}-1} \frac{\sigma_{k-1}(\Lambda)}{\sigma_{l}(\Lambda)}.
\end{aligned}
\end{equation*}
From \eqref{5.53} and \eqref{5.51}, it follows that
\begin{eqnarray*}
\begin{aligned}
 \sum_{i}\tilde{F}^{ii}h_{ii} \geq \tilde{F}^{nn}h_{nn}
 \geq&\frac{(N+1)\kappa_{\min}}{k(N-l+1)}\left[\frac{\sigma_{k}(\Lambda)}{\sigma_{l}(\Lambda)}\right]^{\frac{1}{k-l}-1}\sum_{\{I|n\in I\}}\frac{\sigma_{k-1}(\Lambda|\Lambda_{I})\sigma_{l}(\Lambda|\Lambda_{I})}{\sigma_{l}^{2}(\Lambda)}\\
 \geq&\frac{(N+1)\kappa_{\min}}{k(N-l+1)}\left[\frac{\sigma_{k}(\Lambda)}{\sigma_{l}(\Lambda)}\right]^{\frac{1}{k-l}-1}
 \frac{\sigma_{k-1}(\Lambda|\Lambda_{I_N})\sigma_{l}(\Lambda|\Lambda_{I_N})}{\sigma_{l}^{2}(\Lambda)}\\
 \geq&\frac{(N+1)\kappa_{\min}}{k(N-l+1)}\left[\frac{\sigma_{k}(\Lambda)}{\sigma_{l}(\Lambda)}\right]^{\frac{1}{k-l}-1}
 \frac{\sigma_{k-1}(\Lambda|\Lambda_{I_k})\sigma_{l}(\Lambda|\Lambda_{I_{l+1}})}{\sigma_{l}^{2}(\Lambda)}.\\
 \end{aligned}
\end{eqnarray*}
Using  Proposition~\ref{P1} (\romannumeral4), we derive
$$\sigma_{k-1}(\Lambda|\Lambda_{I_k})\geq c(n,k,\mathbf{p})\sigma_{k-1}(\Lambda), \quad
 \sigma_{l}(\Lambda|\Lambda_{I_{l+1}})\geq c(n,l,\mathbf{p})\sigma_{l}(\Lambda). $$
Combining  the above inequalities with Proposition~\ref{P4}(\romannumeral3),
\begin{eqnarray*}
\begin{aligned}
 \sum_{i}\tilde{F}^{ii}h_{ii}
 \geq&\frac{(N+1)\kappa_{\min}}{k(N-l+1)}\left[\frac{\sigma_{k}(\Lambda)}{\sigma_{l}(\Lambda)}\right]^{\frac{1}{k-l}-1}c(n,k,l,\mathbf{p})\frac{\sigma_{k-1}(\Lambda)}{\sigma_{l}(\Lambda)}\\
 \geq&\frac{(N+1)(k-l)\kappa_{\min}}{\mathbf{p}k(N-k+1)(N-l+1)}c(n,k,l,\mathbf{p}) \sum_{i}\tilde{F}^{ii}\\
  \geq&c_0(\sum_{i}\tilde{F}^{ii}+1).
 \end{aligned}
\end{eqnarray*}
\end{proof}


Next, we prove the upper estimate of double normal second-order derivatives on the boundary as following.
\begin{lemma}\label{c23}
Suppose that $\Omega \subset \mathbb{R}^n$ is a $C^3$  strictly convex domain,
$f$ is a positive function with $f^{\frac1{k-l}}\in C^2(\overline{\Omega})$
 and $u \in C^3(\Omega) \cap C^2(\overline{\Omega})$ is the $(\Lambda, k)$-admissible solution of Hessian quotient equation ~\eqref{(1.7)}, then
\begin{equation*}
\max_{\partial \Omega} u_{\nu\nu} \leq C,
\end{equation*}
where $C$ depends on $n, k, l, \mathbf{p}, \Omega,  |f^{\frac{1}{k-l}}|_{C^1}$ and $|\varphi|_{C^2}$.
\end{lemma}

\begin{proof}
Without loss of generality, we assume $\max_{\partial\Omega} u_{\nu\nu} > 0$. If $\max_{\partial\Omega} u_{\nu\nu} \leq -\min_{\partial\Omega} u_{\nu\nu}$, (i.e. $\max_{\partial\Omega} |u_{\nu\nu}| = -\min_{\partial\Omega} u_{\nu\nu}$), the result follows immediately from Lemma~\ref{c24},
$$\max_{\partial\Omega} u_{\nu\nu} \leq -\min_{\partial\Omega} u_{\nu\nu} \leq C.$$
In the following, we assume $\max_{\partial\Omega} u_{\nu\nu} > -\min_{\partial\Omega} u_{\nu\nu}$. Denote $M =\max_{\partial\Omega} u_{\nu\nu}> 0$ and let ${z}_0 \in \partial\Omega$ such that $\max_{\partial\Omega} u_{\nu\nu} = u_{\nu\nu}({z}_0)$. Motivated by Chen-Zhang~\cite{CZ},  consider the test function
\begin{equation*}\label{c221}
{P}(x) = \big(1 + \beta d(x)\big)
\Big[ Du(x) \cdot (-Dd(x)) + \varepsilon u(x) - \varphi(x) \Big] - \big(A + \frac{1}{2}M \big)h(x),
\end{equation*}
where $\beta$ and $A$ are positive constants to be determined  later.

Assume that  ${P}$  attains its minimum at some point ${x}_0 \in \Omega_\mu$. By rotating the coordinate system, we may assume that $D^2u({x}_0)$
is diagonal. All subsequent computations are carried out at ${x}_0$. First, we differentiate $P$ twice to obtain
\begin{align}
0 = {P}_i 
&= \beta d_i \big( -\sum_j u_j d_j +\varepsilon u - \varphi \big) + (1 + \beta d) \big( -u_{ii}d_i - \sum_j u_j d_{ji} +\varepsilon u_i - \varphi_i\big)\notag  \\
&\quad - (A + \frac{1}{2}M) h_i\label{c222},
 \end{align}
and
\begin{align}
0 \leq  \tilde{F}^{ii} {P}_{ii}
=&\, \beta \tilde{F}^{ii}d_{ii} \big( -\sum_j u_j d_j +\varepsilon u - \varphi \big)
+ 2\beta \tilde{F}^{ii}d_i \big( -u_{ii}d_i- \sum_j u_j d_{ji} + \varepsilon u_i - \varphi_i \big) \notag \\
&+(1 + \beta d)\tilde{F}^{ii} \big( -\sum_j u_{jii}d_j -2u_{ii} d_{ii}- \sum_j u_j d_{jii} +\varepsilon u_{ii} - \varphi_{ii} \big)\notag\\
&- \big(A + \frac{1}{2}M\big)\tilde{F}^{ii} h_{ii} \notag \\
\leq&\, -2\beta  \tilde{F}^{ii} u_{ii} d_i^2 - 2(1 + \beta d)\tilde{F}^{ii} u_{ii} d_{ii}
+ \Bigl[\beta C_3 -\bigl(A + \frac{1}{2}M\bigr) c_0  \Bigr] \Bigl( \sum_{i} \tilde{F}^{ii} + 1 \Bigr), \label{c224}
\end{align}
where $C_3$ is a constant depending on $M_0, M_1, |\varphi|_{C^2}, |f^\frac{1}{k-l}|_{C^1}$ and $\Omega$, note that we have used Lemma \ref{c26} in the last line.

The set of indices $\{1,2,\ldots,n\}$ is divided into two subsets as follows:
$$B = \left\{ i \big|\, \beta d_i^2 < \frac{1}{n} \right\},~~\quad ~G = B^c = \left\{ i \big|\, \beta d_i^2 \geq \frac{1}{n} \right\},$$
and choose $\beta> 1$, $\mu \leq\frac{1}{\beta}$. It implies $\sum_{i \in B} d_i^2 < 1 = |Dd|^2$,
that means $G$ is not empty, there exists an $i_0 \in G$ such that $d_{i_0}^2\geq d_i^2$ for any $i\in G$, then
\begin{equation}\label{c225}
d_{i_0}^2 \geq \frac{1}{n}|Dd|^2 = \frac{1}{n}.
\end{equation}
 For any $i \in G$, by \eqref{c222}, we get
\begin{equation*}
u_{ii} = \frac{(1 - 2d)\left(A + \frac{1}{2}M\right)}{1 + \beta d}  + \frac{\beta \left(-\sum_j u_j d_j + \varepsilon u - \varphi\right)}{1 + \beta d} - \frac{\sum_j u_j d_{ji} -\varepsilon u_i + \varphi_i}{d_i}.
\end{equation*}
Note that $d_i^2 \geq \frac{1}{n\beta}$, and $1<1+\beta d\leq2$, it holds
\begin{eqnarray}\label{c226}
\begin{aligned}
\left| \frac{\beta \left( -\sum_j u_j d_j + \varepsilon u - \varphi \right)}{1 + \beta d} - \frac{\sum_j u_j d_{ji} - \varepsilon u_i + \varphi_i}{d_i} \right|
 \leq \beta C_4,
 \end{aligned}
\end{eqnarray}
where $C_4$ is a constant depending on $M_0, M_1, n, |\varphi|_{C^1}$ and $\Omega$. Let $A \geq 5\beta C_4$, we derive
\begin{equation}\label{c2226}
\frac{A}{5} + \frac{M}{5} \leq u_{ii} \leq \frac{6A}{5} + \frac{M}{2}, \quad \forall~i \in G.
\end{equation}
From \eqref{c225}, assume $i_0=1$, then
\begin{equation}\label{c227}
\begin{split}
-2\beta \sum_{i } \tilde{F}^{ii} u_{ii} d_i^2 &= -2\beta \sum_{i \in G} \tilde{F}^{ii} u_{ii} d_i^2 - 2\beta \sum_{i \in B} \tilde{F}^{ii} u_{ii} d_i^2 \\
&\leq -2\beta \tilde{F}^{11} u_{11} d_1^2 - 2\beta \sum_{u_{ii} < 0, i\in B} \tilde{F}^{ii} u_{ii} d_i^2 \\
&\leq -\frac{2\beta }{n}\tilde{F}^{11} u_{11} - \frac{2 }{n} \sum_{u_{ii} < 0, i\in B} \tilde{F}^{ii} u_{ii}.
\end{split}
\end{equation}
Moreover,
\begin{equation}\label{c228}
\begin{aligned}
-2(1+\beta d) \sum_{i } \tilde{F}^{ii} u_{ii} d_{ii} \leq -2(1+\beta d) \sum_{u_{ii} \geq 0} \tilde{F}^{ii} u_{ii} d_{ii}
\leq 4\kappa_{\max}\sum_{u_{ii} \geq 0} \tilde{F}^{ii} u_{ii}.
 \end{aligned}
\end{equation}
where $\kappa_{\max}$ is defined  in Lemma \ref{c26}.
Plugging \eqref{c227} and \eqref{c228} into \eqref{c224},
\begin{equation}\label{c2228}
\begin{split}
0\leq  \tilde{F}^{ii} {P}_{ii} \leq &-\frac{2\beta }{n}\tilde{F}^{11} u_{11} - \frac{2 }{n} \sum_{u_{ii} < 0} \tilde{F}^{ii} u_{ii} +4\kappa_{\max}\sum_{u_{ii} \geq 0} \tilde{F}^{ii} u_{ii}\\
&+ \Bigl[\beta C_3 -\bigl(A + \frac{1}{2}M\bigr) c_0  \Bigr] \Bigl( \sum_{i} \tilde{F}^{ii} + 1 \Bigr).\\
\end{split}
\end{equation}
We divide the proof into three cases.  Without loss of generality,  assume that $u_{22} \geq u_{33} \geq \cdots \geq u_{nn}$.

\textbf{Case 1: } $u_{nn} \geq 0$. In this case,
\begin{equation}\label{c229}
4\kappa_{\max}  \sum_{u_{ii} \geq 0} \tilde{F}^{ii} u_{ii} = 4\kappa_{\max}  \sum_{i} \tilde{F}^{ii} u_{ii} = 4\kappa_{\max}f^{\frac{1}{k-l}}.
\end{equation}
Based on \eqref{c2228} and \eqref{c229}, we derive
\begin{align*}
0 \leq  \tilde{F}^{ii} {P}_{ii} \leq 4\kappa_{\max}f^{\frac{1}{k-l}}
+ \Bigl[\beta C_3 -\bigl(A + \frac{1}{2}M\bigr) c_0  \Bigr] \Bigl( \sum_{i} \tilde{F}^{ii} + 1 \Bigr)<0,
\end{align*}
if  $A > \dfrac{4\kappa_{\max} \max_{x\in \bar{\Omega}}|f^{\frac{1}{k-l}}| + \beta C_3}{c_0} := A_1$. This is a contradiction.

\textbf{Case 2: }  $u_{nn}<0$ and $\frac{c_0}{3(4\kappa_{\max} + \frac{2}{n})}u_{11} \geq -u_{nn}$. Due to the equation, we obtain
\begin{equation*}
\sum_{u_{ii}<0} \tilde{F}^{ii} u_{ii} + \sum_{u_{ii}\geq 0} \tilde{F}^{ii} u_{ii} = f^{\frac{1}{k-l}}.
\end{equation*}
It follows from \eqref{c2226} that
\begin{equation}\label{c2210}
\begin{aligned}
4\kappa_{\max}\sum_{u_{ii} \geq 0} \tilde{F}^{ii} u_{ii}- \frac{2 }{n} \sum_{u_{ii} < 0} \tilde{F}^{ii} u_{ii}
=&4\kappa_{\max}f^{\frac{1}{k-l}}+(4\kappa_{\max}+\frac{2 }{n})\sum_{u_{ii} < 0} \tilde{F}^{ii} |u_{ii}|\\
\leq &4\kappa_{\max}f^{\frac{1}{k-l}}+(4\kappa_{\max}+\frac{2 }{n})\sum_{u_{ii} < 0} \tilde{F}^{ii} |u_{nn}|\\
\leq &4\kappa_{\max}f^{\frac{1}{k-l}}+\frac{c_0}{3}\sum_i\tilde{F}^{ii}u_{11}\\
\leq &4\kappa_{\max}f^{\frac{1}{k-l}}+c_0\Bigl( \frac{2A}{5} + \frac{M}{6} \Bigr)\sum_i\tilde{F}^{ii}.
 \end{aligned}
\end{equation}
Hence, from \eqref{c2228} and \eqref{c2210},
\begin{equation*}\label{c2211}
\begin{aligned}
  0\leq 4\kappa_{\max}f^{\frac{1}{k-l}}+c_0\Bigl( \frac{2A}{5} + \frac{M}{6} \Bigr)\sum_i\tilde{F}^{ii}
 + \Bigl[\beta C_3 -\bigl(A + \frac{1}{2}M\bigr) c_0  \Bigr] \Bigl( \sum_{i} \tilde{F}^{ii} + 1 \Bigr)
  < 0,
 \end{aligned}
\end{equation*}
if we choose $A > \dfrac{5\big(4\kappa_{\max} \max_{x\in \bar{\Omega}}|f^{\frac{1}{k-l}}| +\beta C_3\big)}{3c_0} := A_2$. This is a contradiction.

\textbf{Case 3: } $\frac{c_0}{3(4\kappa_{\max} + \frac{2}{n})}u_{11}< -u_{nn}$. In this case, we have $u_{11} \geq \frac{A}{5} + \frac{M}{5}$ by \eqref{c2226}, and $u_{22} \leq M_2(1 + M)$ from Theorem \ref{C211}. Therefore,
\begin{equation*}
u_{11} \geq \frac{1}{5M_2} u_{22}.
\end{equation*}
By Lemma \ref{l2}, we have
\begin{equation}\label{c2212}
\tilde{F}^{11} \geq C_2 \sum_{i} \tilde{F}^{ii},
\end{equation}
where $C_2$ depending on $M_2, \kappa_{\max}, n, k, l, \mathbf{p}, c_0$.
From \eqref{c2228} and \eqref{c2212},
\begin{align*}
0 \leq \tilde{F}^{ii} {P}_{ii}
\leq &-\frac{2\beta }{n} \Bigl(\frac{A}{5} + \frac{M}{5}\Bigr)C_2 \sum_{i} \tilde{F}^{ii}+4\kappa_{\max}\sum_{u_{ii} \geq 0} \tilde{F}^{ii} u_{ii}- \frac{2 }{n}\Bigl( f^{\frac{1}{k-l}}-\sum_{u_{ii}\geq 0} \widetilde{F}^{ii} u_{ii}\Bigr)\notag\\
&+ \Bigl[\beta C_3-\bigl(A + \frac{1}{2}M\Bigr) c_0  \bigr] \Bigl( \sum_{i} \tilde{F}^{ii} + 1 \Bigr)\notag\\
\leq &-\frac{2\beta }{n} \Bigl(\frac{A}{5} + \frac{M}{5}\Bigr)C_2 \sum_{i} \tilde{F}^{ii}+\Bigl(4\kappa_{\max}+\frac{2 }{n}\Bigr)M_2(1 + M)\sum_{i} \tilde{F}^{ii}\notag \\
&+ \Bigl[\beta C_3-\bigl(A + \frac{1}{2}M\bigr) c_0  \Bigr] \Bigl( \sum_{i} \tilde{F}^{ii} + 1 \Bigr)\notag\\
\leq &\Bigl[-\frac{2\beta C_2}{5n}A+ \Bigl(4\kappa_{\max}+\frac{2 }{n}\Bigr)M_2 \Bigr]\sum_{i} \tilde{F}^{ii}
+ \bigl(\beta C_3-A  c_0  \bigr) \Bigl( \sum_{i} \tilde{F}^{ii} + 1 \Bigr)\notag\\
&+\Bigl[-\frac{2\beta C_2}{5n}+ \Bigl(4\kappa_{\max}+\frac{2 }{n}\Bigr)M_2 \Bigr]M\sum_{i} \tilde{F}^{ii}\notag\\
<&~0\label{c2213},
\end{align*}
if $\beta > \dfrac{5nM_2 (2\kappa_{\max} + \frac{1}{n})}{C_2} $ and $A > 1+\frac{\beta C_3}{c_0}:=A_3 $. This is a contradiction.

Therefore ${P}$ attains its minimum on $\partial\Omega_\mu$. Moreover, on $\partial\Omega$,  $d = h = 0$, and $-Dd = \nu$, it is easy to see
$$    P = 0.$$
On $\partial\Omega_\mu \cap \Omega$,  $d = \mu$. Thus,
\begin{equation*}
P\geq -C_5(|u|_{C^1}, |\varphi(x)|_{C^0})  + \frac{9}{10}\mu A  > 0,
\end{equation*}
if we take $\beta = \frac{5nM_1\left(2\kappa_{\max} + \frac{1}{n}\right)}{C_2}+1$, $\mu = \min\{\frac{1}{10}, \frac{1}{\beta}\}$,  and
 $$A = \max\left\{A_1, A_2, A_3,5\beta C_4, \frac{10C_5}{9\mu}\right\}+1.$$
Finally the maximum principle tells us that $P \geq 0$ in  $\Omega_\mu$, and attains its minimum only on $\partial\Omega$.
Suppose $u_{\nu\nu}(z_0) = \max_{\partial\Omega} u_{\nu\nu}=M > 0$,
\begin{align*}
    0 &\geq P_\nu(z_0) \\
    &=\Bigl(u_{\nu\nu}(z_0)-\sum_ju_jd_{j\nu}+\varepsilon u_\nu-\varphi_\nu\Bigr)-\bigl(A + \frac{1}{2}M\bigr) h_\nu\\
    &\geq M - C - \bigl(A +\frac{1}{2}M\bigr).
\end{align*}
Then we obtain
\begin{equation*}
    \max_{\partial\Omega} u_{\nu\nu} \leq C.
\end{equation*}
The proof of Lemma~\ref{c23} is complete.

\end{proof}

Finally, we will prove the lower estimate of double normal second-order derivatives on the boundary.

\begin{lemma} \label{c24}
Let $\Omega \subset \mathbb{R}^n$ be a $C^3$ strictly convex domain, $f$ be  a positive function with $f^{\frac1{k-l}}\in C^2(\overline{\Omega})$ and $\varphi \in C^3(\partial \Omega)$. Suppose  $u \in C^3(\Omega) \cap C^2(\overline{\Omega})$ is the $(\Lambda, k)$-admissible solution of Hessian quotient equation ~\eqref{(1.7)}, then
\begin{equation*}
\min_{\partial \Omega} u_{\nu\nu} \geq -C,
\end{equation*}
where $C$ is a positive constant depending on $n, k, l,  \mathbf{p}, \Omega, |\varphi|_{C^2}, |f^\frac{1}{k-l}|_{C^1}$.
\end{lemma}

\begin{proof}
Without loss of generality, we assume $\min_{\partial\Omega} u_{\nu\nu} < 0$. If $-\min_{\partial\Omega} u_{\nu\nu} < \max_{\partial\Omega} u_{\nu\nu}$, (i.e. $\max_{\partial\Omega} |u_{\nu\nu}| = \max_{\partial\Omega} u_{\nu\nu}$), the result follows immediately from Lemma~\ref{c23},
$$-\min_{\partial\Omega} u_{\nu\nu} < \max_{\partial\Omega} u_{\nu\nu} \leq C.$$
In the following, we assume $-\min_{\partial\Omega} u_{\nu\nu} \geq \max_{\partial\Omega} u_{\nu\nu}$, that is $\max_{\partial\Omega} |u_{\nu\nu}| = -\min_{\partial\Omega} u_{\nu\nu}$. Denote $\bar{M} = -\min_{\partial\Omega} u_{\nu\nu} > 0$ and let $\bar{z}_0 \in \partial\Omega$ such that $\min_{\partial\Omega} u_{\nu\nu} = u_{\nu\nu}(\bar{z}_0)$.
 In a similar way, we construct the test function as
\begin{equation*}\label{c231}
\bar{P}(x) = \big(1 + \beta d(x)\big)
\Big[ Du(x) \cdot (-Dd(x)) + \varepsilon u(x) - \varphi(x) \Big] + \big(A + \frac{1}{2}M \big)h(x),
\end{equation*}
where $\beta$ and $A$ are positive constants to be determined  later.

Assume that $\bar{P}$  attains its maximum at some point $\bar{x}_0 \in \Omega_\mu$. By rotating the coordinate system, we may assume that $D^2u({\bar{x}}_0)$
is diagonal. All subsequent computations are carried out at $\bar{x}_0$. First, we differentiate $\bar{P}$ twice to obtain
\begin{align}
0 = \bar{{P}}_i
&= \beta d_i \big( -\sum_j u_j d_j + \varepsilon u - \varphi \big) + (1 + \beta d) \big( -u_{ii}d_i - \sum_j u_j d_{ji} +\varepsilon u_i - \varphi_i \big) \notag \\
&\quad + \big(A + \frac{1}{2}\bar{M}\big) h_i\label{c232},
 \end{align}
and
\begin{align}
0 \geq  \tilde{F}^{ii} \bar{{P}}_{ii}
=&\, \beta \tilde{F}^{ii}d_{ii} \big( -\sum_j u_j d_j +\varepsilon u - \varphi \big)
+ 2\beta \tilde{F}^{ii}d_i \big( -u_{ii}d_i- \sum_j u_j d_{ji} + \varepsilon u_i - \varphi_i \big) \notag \\
&+(1 + \beta d)\tilde{F}^{ii} \big( -\sum_j u_{jii}d_j -2u_{ii} d_{ii}- \sum_j u_j d_{jii} +\varepsilon u_{ii} - \varphi_{ii} \big)\notag\\
&+ \big(A + \frac{1}{2}\bar{M}\big)\tilde{F}^{ii} h_{ii} \notag \\
\geq&\, -2\beta  \tilde{F}^{ii} u_{ii} d_i^2 - 2(1 + \beta d)\tilde{F}^{ii} u_{ii} d_{ii}
+ \Bigl[\bigl(A + \frac{1}{2}\bar{M}\bigr) c_0 -\beta C_6 \Bigr] \Bigl( \sum_{i} \tilde{F}^{ii} + 1 \Bigr), \label{c234}
\end{align}
where $C_6$ is a constant depending on $M_0, M_1, |\varphi|_{C^2}, |f^\frac{1}{k-l}|_{C^1}$ and $\Omega$, note that we have used Lemma \ref{c26} in the last line.

The set of indices $\{1,2,\ldots,n\}$ is divided into two subsets as follows:
$$B = \left\{ i \big|\, \beta d_i^2 < \frac{1}{n} \right\},~~ \quad ~G = B^c = \left\{ i \big|\, \beta d_i^2 \geq \frac{1}{n} \right\},$$
and choose $\beta> 1$, $\mu \leq\frac{1}{\beta}$. It implies $\sum_{i \in B} d_i^2 < 1 = |Dd|^2$, that means $G$ is not empty, there exists an $i_0 \in G$ such that $d_{i_0}^2\geq d_i^2$ for any $i\in G$, then
\begin{equation}\label{c235}
d_{i_0}^2 \geq \frac{1}{n}|Dd|^2 = \frac{1}{n}.
\end{equation}
 For any $i \in G$, it follows from \eqref{c232} that
\begin{equation*}
u_{ii} = \frac{( 2d-1)\left(A + \frac{1}{2}\bar{M}\right)}{1 + \beta d}  + \frac{\beta \left(-\sum_j u_j d_j + \varepsilon u - \varphi\right)}{1 + \beta d} - \frac{\sum_j u_j d_{ji} -\varepsilon u_i + \varphi_i}{d_i}.
\end{equation*}
 From \eqref{c226}, it holds
\begin{equation*}\label{c236}
\begin{aligned}
\left| \frac{\beta \left( -\sum_j u_j d_j + \varepsilon u - \varphi \right)}{1 + \beta d} - \frac{\sum_j u_j d_{ji} - \varepsilon u_i + \varphi_i}{d_i} \right|  \leq \beta C_4,
 \end{aligned}
\end{equation*}
where $C_4$ is a constant depending on $M_0, M_1, n, |\varphi|_{C^1}$ and $\Omega$. Let $A \geq 5\beta C_4$, then
\begin{equation*}
-\frac{6}{5}A - \frac{1}{2}\bar{M} \leq u_{ii} \leq -\frac{1}{5}A - \frac{1}{5}\bar{M}, \quad \forall~i \in G.
\end{equation*}
By \eqref{c235}, we may assume $i_0=1$,
\begin{equation}\label{c237}
\begin{aligned}
-2\beta \sum_{i} \tilde{F}^{ii} u_{ii} d_i^2 &= -2\beta \sum_{i \in G} \tilde{F}^{ii} u_{ii} d_i^2 - 2\beta \sum_{i \in B} \tilde{F}^{ii} u_{ii} d_i^2 \\
&\geq -2\beta \tilde{F}^{11} u_{11} d_1^2 - 2\beta \sum_{u_{ii} \geq 0, i \in B} \tilde{F}^{ii} u_{ii} d_i^2 \\
&\geq -\frac{2\beta }{n}\tilde{F}^{11} u_{11} - \frac{2 }{n} \sum_{u_{ii} \geq 0} \tilde{F}^{ii} u_{ii}\\
&=-\frac{2\beta }{n}\tilde{F}^{11} u_{11}-\frac{2 }{n}\Bigl( f^{\frac{1}{k-l}}-\sum_{u_{ii}<0} \widetilde{F}^{ii} u_{ii}\Bigr),
 \end{aligned}
\end{equation}
and
\begin{equation}\label{c238}
\begin{aligned}
-2(1+\beta d) \sum_{i } \tilde{F}^{ii} u_{ii} d_{ii} \geq  -2(1+\beta d) \sum_{u_{ii} < 0} \tilde{F}^{ii} u_{ii} d_{ii}
\geq 4\kappa_{\max}\sum_{u_{ii} < 0} \tilde{F}^{ii} u_{ii},
 \end{aligned}
\end{equation}
where $\kappa_{\max}$ is defined  in Lemma \ref{c26}.
Substituting \eqref{c237} and \eqref{c238} into \eqref{c234} yields
\begin{equation*}\label{c2338}
\begin{aligned}
0 \geq &-\frac{2\beta }{n}\tilde{F}^{11} u_{11}  +\Bigl(4\kappa_{\max}+ \frac{2 }{n}\Bigl)\sum_{u_{ii} < 0} \tilde{F}^{ii} u_{ii}-\frac{2 }{n} f^{\frac{1}{k-l}}\\
&+ \Bigl[\bigl(A + \frac{1}{2}\bar{M}\bigr) c_0 - \beta C_6 \Bigr] \Bigl( \sum_{i} \tilde{F}^{ii} + 1 \Bigr).
 \end{aligned}
\end{equation*}
Because $u_{11} < 0$,  by Lemma~\ref{f11} we obtain
$$\tilde{F}^{11} \geq C_1\sum_{i} \tilde{F}^{ii}.$$
Furthermore, together with Theorem \ref{C211},
\begin{equation*}
\begin{aligned}
0\geq &~\frac{2\beta }{n}\Bigl(\frac{1}{5}A + \frac{1}{5}\bar{M}\Bigr)C_1\sum_{i}\tilde{F}^{ii}- \Bigl(4\kappa_{\max}+ \frac{2 }{n}\Bigr)M_2(1+\bar{M})  \sum_{i} \tilde{F}^{ii}\\
 &+ \Bigl[\bigl(A + \frac{1}{2}\bar{M}\bigr) c_0 - \beta C_6-\frac{2 }{n} f^{\frac{1}{k-l}} \Bigr] \Bigl( \sum_{i} \tilde{F}^{ii} + 1 \Bigr)\\
>&~0,
\end{aligned}
\end{equation*}
if we choose $\beta > \dfrac{5n(2\kappa_{\max}+ \frac{1 }{n})M_2}{C_1} $ and $A > \dfrac{\beta C_6}{c_0} +\dfrac{\frac{2 }{n}\max_{x\in \bar{\Omega}}|f^{\frac{1}{k-l}}|}{c_0}+ 1$.
This contradicts to that $\bar{P}$ attains its maximum in the interior of $\Omega_\mu$. It implies that $\bar{P}$ attains its maximum on the boundary $\partial\Omega_\mu$. On $\partial\Omega$,  $d = h = 0$, and $-Dd = \nu$, it is easy to see
$$    \bar{P} = 0.$$
On $\partial\Omega_\mu \cap \Omega$,  $d = \mu$, thus
\begin{equation*}
    \bar{P}\leq C_7(|u|_{C^1}, |\varphi(x)|_{C^0})  - \frac{9}{10}\mu A  < 0.
\end{equation*}
if we take $\beta =\frac{5n(2\kappa_{\max}+ \frac{1 }{n})M_2}{C_1}+1$, $\mu = \min\{\frac{1}{10}, \frac{1}{\beta}\}$,  and
$$A = \max\left\{5\beta C_4, \frac{\beta C_6}{c_0} +\frac{\frac{2 }{n}\max_{x\in \bar{\Omega}}|f^{\frac{1}{k-l}}|}{c_0}+ 1, \frac{10C_7}{9\mu}\right\}+1.$$
Finally the maximum principle tells us that $\bar{P}$ in  $\Omega_\mu$, and attains its maximum only on $\partial\Omega$.
Suppose $u_{\nu\nu}(\bar{z}_0) = \min_{\partial\Omega} u_{\nu\nu}=\bar{M}$, then
\begin{equation*}
\begin{aligned}
    0 \leq& \bar{P}_\nu(\bar{z}_0) \\
    =& \Bigl[ u_{\nu\nu}(\bar{z}_0) - \sum_j u_j d_{j\nu} +\varepsilon u_\nu - \varphi_\nu \Bigr] + \Bigl( A + \frac{1}{2}\bar{M} \Bigr)\\
      \leq& \bar{M} + C +A +\frac{1}{2}\bar{M}.
\end{aligned}
\end{equation*}
Therefore,
\begin{equation*}
    \min_{\partial\Omega} u_{\nu\nu} \geq -C.
\end{equation*}
The proof of Lemma~\ref{c24} is complete.
\end{proof}

Combining with  Theorem \ref{C211}, Theorem \ref{C22}, Lemma \ref{c23} and Lemma \ref{c24},  we have
\begin{theorem}\label{c25}
 Let $\Omega \subset \mathbb{R}^n$ be a $C^3$  strictly convex domain, $f$ be a positive function with $f^{\frac1{k-l}}\in C^2(\overline{\Omega})$ and $\varphi \in C^3(\partial \Omega)$. Suppose $u \in C^3(\Omega) \cap C^2(\overline{\Omega})$ is the $(\Lambda, k)$-admissible solution of Hessian quotient equation ~\eqref{(1.7)}, then
$$\underset{\overline \Omega}{\max}|D^{2}u|\leq C,$$
where~$C$~is a positive constant depending on $n, k, l, p, |\varphi|_{C^3}, |f^{\frac{1}{k-l}}|_{C^{2}}$  and $ \Omega$.
\end{theorem}

\section{Existence of the Neumann Boundary Problem}
In this section, we complete the proof of Theorem \ref{T1} and Theorem \ref{T2}.

\begin{proof}[\textbf{Proof of Theorem \ref{T1}.}]   Let $f_{j}\in C^{\infty}(\overline{\Omega})$ be positive functions such that \mbox{$|f_{j}^{\frac1{k-l}}|_{C^2(\Omega)}\leq C$}, and \mbox{$\lim_{j \to \infty} |f_j - f|_{C^2(\overline{\Omega})} = 0$}, where the constant $C$ only depends on $|f^{\frac1{k-l}}|_{C^2(\overline{\Omega})}$. Consider  the  Neumann problem for the following approximation equation
\begin{eqnarray}\label{(6.1)}
\begin{cases}
\frac{\sigma_k(\Lambda(D^2u_{s}^j))}{\sigma_l(\Lambda(D^2u_{s}^j))}=f_j(x)&\text{in}~ \Omega,\\\\
(u_{s}^j)_\nu=-\varepsilon_su_{s}^j+\varphi&\text{on}~ \partial\Omega.
\end{cases}
\end{eqnarray}
for  any $\varepsilon _{s}\in ( 0, 1)$.
Using Theorems \ref{C0}, \ref{C1} and \ref{c25}, we derive the global $C^2$ a priori estimates. Hence,  equation \eqref{(6.1)} is uniformly elliptic in $\overline{\Omega}$. As in \cite{LION}, one obtains global H\"{o}lder estimates for the second derivatives of solutions to  \eqref{(6.1)}. Applying the method of continuity,  there exists a $(\Lambda, k)$-admissible solution $u_{s}^j\in C^{3, \alpha} (\overline{\Omega})$ satisfies  the Neumann problem of  equation \eqref{(6.1)}.
 We also obtain a uniform bound (independent of  $\varepsilon_{s}$)  for the $\varepsilon_su_s^j$, $Du_{s}^{j}$ and $D^{2}u_{s}^{j}$. Then we set $v_{s}^{j}=u_{s}^{j}-\frac{1}{|\Omega|}\int_{\Omega}u_{s}^{j}dx$, then $v_{s}^{j}\in C^{2,\alpha}(\overline{\Omega})$ satisfies the following Neumann
boundary problem:
\begin{eqnarray*}\label{(6.2)}
\begin{cases}
\frac{\sigma_k(\Lambda(D^2v_s^j))}{\sigma_l(\Lambda(D ^2v_s^j))}=f_j(x)&\text{in}~\Omega,\\\\
(v_s^j)_\nu=-\varepsilon_sv_s^j-\frac{1}{|\Omega|}\int_\Omega\varepsilon_su_s^jdx+\varphi&\text{on}~\partial\Omega.
\end{cases}
\end{eqnarray*}

There is a constant $\beta_{j}$ and a function $v^{j}\in C^{2}(\overline{\Omega})$ such that $-\varepsilon_{s}u_{s}^{j}\rightarrow\beta_{j}$, $-\varepsilon_{s}v_{s}^{j}\to0$, $-{\frac{1}{|\Omega|}}\int_{\Omega}\varepsilon_{s}u_{s}^{j}dx\to\beta_{j}$ and $v_{s}^{j}\longrightarrow v^{j}$ uniformly in $C^{2}(\overline{\Omega})$ as $\varepsilon_{s}\rightarrow0$. It is easy to verify that $v^{j}$ is the  $(\Lambda, k)$-admissible solution of
\begin{eqnarray*}\label{(6.3)}
\begin{cases}
\frac{\sigma_k(\Lambda(D^2v^j))}{\sigma_l(\Lambda(D ^2v^j))}=f_j(x)&\text{in}~\Omega,\\\\
v^j_\nu=\beta_j+\varphi(x)&\text{on}~\partial\Omega.
\end{cases}
\end{eqnarray*}

And we can get the results that $|\beta_{j}|$ and $|v^j|_{C^2(\overline{\Omega})}$ can be controlled by a constant $C$ which depends on $n, k, l, \mathbf{p}, \Omega, |\varphi|_{C^3}, |f^{\frac1{k-l}}|_{C^2}$. By the compactness argument, we know that there exists a function $v\in C^{1,1}(\overline{\Omega})$ and a constant $c$ such that $\lim_{j\to+\infty}|v^j-v|_{C^1(\bar\Omega)}=0$ and $\lim_{j\to+\infty}\left|\beta_{j}-c\right|=0$.

Let $j\rightarrow+\infty$, then $v$ is a  $(\Lambda, k)$-admissible  solution of the following Neumann boundary problem
\begin{eqnarray*}\label{(6.4)}
\begin{cases}
\frac{\sigma_k(\Lambda(D^2v))}{\sigma_l(\Lambda(D^2v))}=f(x)&\text{in}~\Omega,\\\\
v_\nu=c+\varphi(x)&\text{on}~\partial\Omega.
\end{cases}
\end{eqnarray*}
Uniqueness follows from the comparison principle for fully nonlinear degenerate elliptic equations (see \cite{ST}) together with the Hopf lemma.
This completes the proof of Theorem \ref{T1}.
\end{proof}
\begin{proof}[\textbf{Proof of Theorem \ref{T2}.}]
By taking $\varepsilon = 1$, the global $C^2$ a priori estimates for equation~\eqref{(1.8)} follow from Theorems~\ref{C0}, \ref{C1}, and~\ref{c25}. The existence and higher-order regularity of the solution then follow from the method of continuity and standard elliptic regularity theory. We omit the details.
\end{proof}

\

\end{document}